\documentclass[11pt]{amsart}

\usepackage{latexsym}
\usepackage{fullpage}
\usepackage{amssymb}
\usepackage{amscd}
\usepackage{float}
\usepackage{graphicx}
\usepackage{amsfonts}
\usepackage{pb-diagram}
 \usepackage{amsmath,amscd}
\usepackage{color,graphicx,times}
\usepackage{xcolor}

\theoremstyle{plain}
\newtheorem{theorem}{Theorem}[section]

\newtheorem{proposition}[theorem]{Proposition}
\newtheorem{lemma}[theorem]{Lemma}

\theoremstyle{definition}
\newtheorem{definition}[theorem]{Definition}
\newtheorem{example}[theorem]{Example}

\newtheorem{remark}[theorem]{Remark}
\newtheorem{remarks}[theorem]{Remarks}
\newtheorem{note}[theorem]{Note}

\begin{document}

\title[Bonded knots \& braids]
  {Topology and Algebra of Bonded Knots and Braids}

\author{Ioannis Diamantis}
\address{Department of Data Analytics and Digitalisation,
Maastricht University, School of Business and Economics,
P.O.Box 616, 6200 MD, Maastricht,
The Netherlands.}
\email{i.diamantis@maastrichtuniversity.nl}

\author{Louis H. Kauffman}
\address{Department of Mathematics, Statistics and Computer Science
University of Illinois at Chicago, 851 South Morgan Street
Chicago, IL, 60607-7045.}
\address{International Institute for Sustainability with Knotted Chiral Meta Matter (WPI-SKCM2),
    Hiroshima University, 1-3-1 Kagamiyama, Higashi-Hiroshima, Hiroshima 739-8526, Japan.}
\email{kauffman@math.uic.edu}
\urladdr{http://www.math.uic.edu/~kauffman/}

\author{Sofia Lambropoulou}
\address{School of Applied Mathematical and Physical Sciences, National Technical University of Athens, Zografou campus, GR-15780 Athens, Greece.}
\email{sofia@math.ntua.gr}
\urladdr{http://www.math.ntua.gr/~sofia}



\setcounter{section}{-1}

\begin{abstract}
In this paper we present a detailed study of \emph{bonded knots} and their related structures, integrating recent developments into a single framework. Bonded knots are classical knots  endowed with embedded bonding arcs modeling physical or chemical bonds. We consider bonded knots in three categories (long, standard, and tight) according to the type of bonds, and in two categories, topological vertex and rigid vertex, according to the allowed isotopy moves, and we define invariants for each category. We then develop the theory of \emph{bonded braids}, the algebraic counterpart of bonded knots. We define the {\it bonded braid monoid}, with its generators and relations, and formulate the analogues of the Alexander and Markov theorems for bonded braids, including an $L$-equivalence for bonded braids. Next, we introduce \emph{enhanced bonded knots and braids}, incorporating two types of bonds (attracting and repelling) corresponding to different interactions. We define the enhanced bonded braid group and show how the bonded braid monoid embeds into this group. Finally, we study \emph{bonded knotoids}, which are open knot diagrams with bonds, and their closure operations, and we define the \emph{bonded closure}. We introduce  \emph{bonded braidoids} as the algebraic counterpart of bonded knotoids. These models capture the topology of open chains with inter and intra-chain bonds and suggest new invariants for classifying biological macromolecules.
\end{abstract}

\maketitle

\section{Introduction}

In this paper, we develop  the theory of {\it bonded knots, braids, knotoids and braidoids}. For illustrations of these objects see  Figures~\ref{bknotbr}  and \ref{bkntooidbr}. These structures reflect different physical situations which can occur in applications, such as protein folding and molecular biology.  We consider three types of bonds: {\it long bonds} which can be knotted or linked and are not local, {\it standard bonds} which are in the form of straight segments and do not cross between themselves, and {\it tight bonds} which occur locally and are nearly equivalent to graphical nodes in the mathematical formalism. 
 We also discuss two kinds of isotopy for the three categories of bonded knots and links, long, standard, and tight, which reflect different physical assumptions about the flexibility of bonds: {\it topological vertex isotopy}, whereby the nodes of the bonds can move freely, and {\it rigid vertex isotopy}, whereby the nodes of the bonds  move along with rigid 3-balls in which they are embedded. 

It is a theme of this paper to compare long bonds, standard bonds and tight bonds in the topological  and rigid  categories. We present the set of allowed moves for each category and highlight  forbidden moves that distinguish bonded links from related concepts like tied links. We also recall and extend invariants for bonded knots: notably, the {\it unplugging technique} for the topological vertex category, the {\it tangle insertion} for the rigid vertex category, and a Kauffman bracket type polynomial (the \emph{bonded bracket}) that is invariant under regular isotopy of rigid vertex bonded links. See \cite{K,KSWZ}.

Continuing this theme we introduce and study  {\it bonded braids} as the algebraic counterpart of bonded links. We establish the {\it bonded braid monoid}, we point out its relation to the singular braid monoid \cite{Baez} and we  extend it to the {\it bonded braid group}.  We further establish the interaction between oriented topological bonded links and bonded braids by means of a closure operation, a  {\it  bonded braiding algorithm} (bonded analogue of the classical Alexander theorem) and  {\it bonded $L$-move equivalences } as  bonded analogues of the classical Markov theorem. The classical Alexander and Markov  theorems \cite{A,M} relate  knots and links to braids and the Artin braid groups, and translate their isotopy to an  algebraic equivalence among braids. So, the algebraic structure of braids together with the two theorems above  furnish the necessary basis for encoding knotted objects by words in braid generators and for the potential construction of knot invariants using algebraic tools. A great paradigm of this approach is the breakthrough work of V.F.R. Jones in the construction of the Jones polynomial (see\cite{Jo} and references therein), the ambient isotopy equivalent to the Kauffman bracket polynomial.  The $L$-move equivalent formulation \cite{LR1} of the Markov theorem provides a geometric as well as algebraic approach to the classical braid  equivalence.

\begin{figure}[H]
\begin{center}
\includegraphics[width=5in]{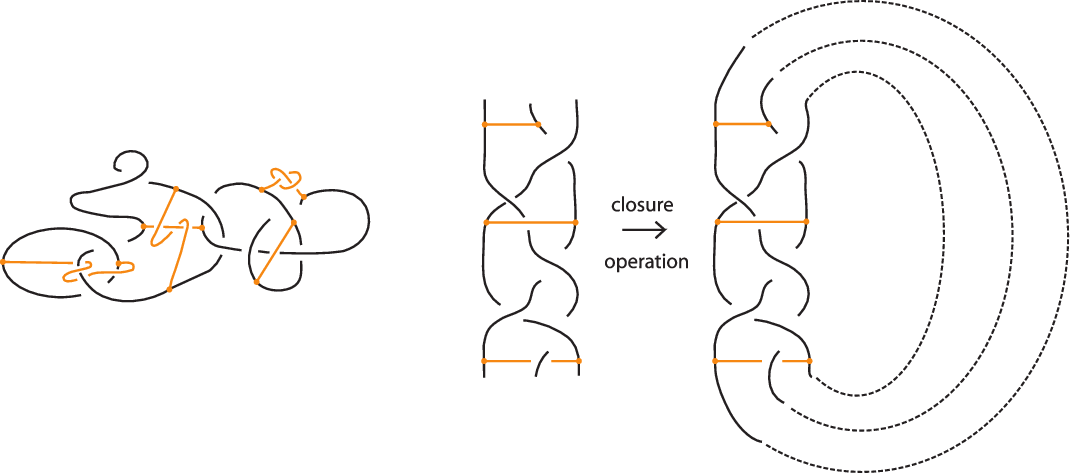}
\end{center}
\caption{On the left: a bonded link. On the right: a bonded braid and its closure.}
\label{bknotbr}
\end{figure}  

Knotted structures in proteins and other biomolecules have become a subject of intense study in recent decades. While only a small percentage of known protein structures are knotted, such knotted proteins provide important insights into protein folding and stability. Traditional knot theory requires closed loops, so to analyze an open protein chain (a polypeptide backbone) one must either artificially connect its ends or adopt new frameworks. Early approaches closed the protein chain by joining its ends in space (for example, via a direct segment connection) and then determined the knot type of the resulting closed loop. This approach can be sensitive to how the closure is performed. 

More recently, \emph{knotoids} were introduced by Turaev in 2012 \cite{T}: a knotoid is essentially a knot diagram with two free ends, considered up to an appropriate equivalence. The theory of knotoids allows one to study the topology of open curves without requiring closure, and indeed the theory of knotoids in $S^2$ (the 2-sphere) extends classical knot theory by distinguishing different ways an open curve can exhibit knotting. Knotoids have been first applied to model open knotted protein chains without artificial closure in~\cite{GDBS}. 

\begin{figure}[H]
\begin{center}
\includegraphics[width=4.4in]{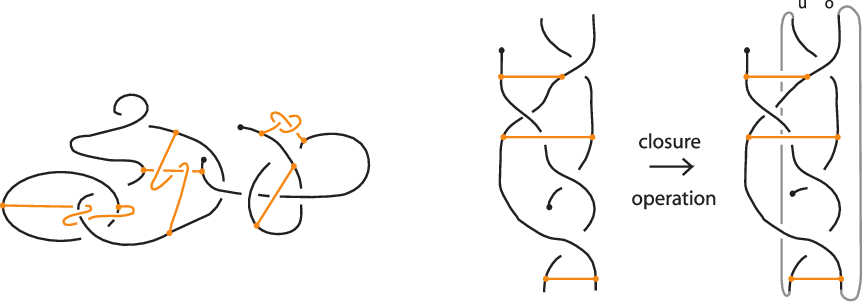}
\end{center}
\caption{On the left: a bonded knotoid. On the right: a bonded braidoid and its closure}
\label{bkntooidbr}
\end{figure}

While knotoids address the issue of open endpoints, real biomolecules such as proteins often have additional internal connections that influence their topology. For instance,  proteins form \emph{bonds} such as disulfide bridges, hydrogen bonds, or salt bridges, which link distant parts of the chain. These bonds can create topologically non-trivial structures (e.g. forming loops or lassos) even if the protein backbone is not closed \cite{KNN, SG, MTM, DSC}. A classical knot or knotoid representation of a protein does not capture the presence of such bonds. To incorporate these features, Kauffman, Magarshak \cite{KM}  and others \cite{TLKL} developed new knot invariants that take into account ``stem'' regions or paired interactions in biomolecules. In particular, studies of RNA tertiary structure and protein folding in the 1990s and 2000s led to polynomial invariants sensitive to base pair bonds or disulfide bonds. 

These developments eventually led to the notion of a \emph{bonded knotoid} \cite{GGLDSK} and more recently of a \emph{bonded link}, which is a knotoid or link augmented by embedded arcs representing bonds between points on the strands.
 Bonded knotoids or  links were formalized as topological models for protein chains with bridges in \cite{TLKL, GGLDSK} and subsequent works \cite{Ad,G}. In a bonded  link or knotoid, the usual  link or knotoid diagram is supplemented with one or more distinguished arcs, called \emph{bonds}, connecting points on the knotoid/link strands (the bonds are drawn as, say, orange arcs in diagrams). These bonds represent physical connections that are not part of the covalently linked backbone. For example, a disulfide bond in a protein can be modelled as a bond connecting two cysteine points on the chain \cite{DSC}.

Mathematically, a bond is an embedded arc whose endpoints lie on (distinct or possibly the same) components of the link diagram. The incorporation of bonds into knot diagrams gives rise to a rich extension of knot theory, with new moves and new invariants. A number of recent papers have studied invariants of bonded links and knotoids, with  applications to protein structure classification, providing a novel way to distinguish different folding patterns of open proteins, see for example \cite{KM,GGLDSK,Ad,G}.

 In this paper we focus on bonded knots/knotoids and bonded braids/braidoids. Our motivation for exploring the diagrammatic settings and equivalences of bonded knots/knotoids  comes mainly from the remarkable applications so far and their interest as mathematical objects. Then, a new diagrammatic setting leads naturally to the fundamental questions about existence or not of related braid structures. The algebraic structure of bonded braids can be used for encoding the  objects that they model with words in the generators, after applying a braiding algorithm for turning them isotopically into closed bonded braids. Topological vertex isotopy translates into  a bonded braid equivalence generated by moves between bonded braids. One can exploit algebraic tools for constructing topological invariants in order to distinguish our topological objects. The $L$-moves generating the classical braid equivalence are  fundamental, so that they provide an adaptive frame for formulating braid equivalences in other diagrammatic settings.
 For ensuring a  sufficient set of moves for generating a braid equivalence one has to examine all algorithmic choices and moves in the diagrammatic setting, and finding all of them can be very subtle. See for example \cite{LR1, La, KaLa} for different diagrammatic settings.  This is also the case here for pinning down the {\it bonded $L$-moves}, which augment the classical $L$-moves for generating the bonded braid equivalence. We further describe an algorithm for turning a bonded knotoid into a bonded braidoid and we also extend the concept of bond  by introducing two types of mutually cancelling bonds, the {\it enhanced bonds}. 
 The definition and study of bonded braids/braidoids and their relation to bonded knots/knotoids form the core of this work.

In regard to interdisciplinary connections,  we note that the interaction lines in Feynman diagrams have the formal structure of bonds in our sense. Thus a Feynman diagram can be regarded as a bonded graph (possibly with a knotted embedding in three-space). In fact, just such ideas are in back of the work of Kreimer in the book “Knots and Feynman Diagrams” \cite{Kr}  where the bonds in the Feynman diagrams undergo tangle insertion (in our sense) and are thereby associated with specific knots and links. Kreimer suggests that the topological types of these knots and links associated with the diagrams are significantly related to the physical evaluations of the diagrams.

In another direction, we note that current suggestions about knotted glueballs (closed loops of gluon flux related to the structure of protons) can be seen in our context as knotted structures consisting entirely of bonds. The strings of gluon field form highly attracting bonds between quarks in this model \cite{N,KB}. Finally, we point out that there is ongoing research in the interface of molecular biology and the construction of molecules with specified polyhedral shapes \cite{NJ}. This research also involves bonds to which our modelling applies.
The paper is structured as follows. In Section~\ref{sec2:bondedlinks}, we develop the notion of {\it long bonded links} and we establish the allowed isotopy moves in both topological and rigid vertex categories. In Section~\ref{sec3:stbondedlinks}, we introduce {\it standard bonded links}, namely, long bonded links with unknotted and unlinked bonds, and we establish that any long bonded link can be isotoped into standard form. We also give a full set of isotopy moves for standard bonded links in both topological and rigid vertex categories. We also identify some forbidden moves due to bonds. In Section~\ref{sec4:tightbondedlinks}, we introduce a stricter category of bonded links, called {\it tight bonded links}. Tight bonded links are standard bonded links whose bonds do not interact with link arcs. We also present a set of isotopy moves for this setting. 

In Section~\ref{inv} we construct invariants of long, standard and tight bonded links via the unplugging operation for the topological category and the tangle insertion technique for the rigid vertex category. Furthermore, we extend the bracket polynomial for rigid tight bonded links. We note that all three categories of bonded links (long, standard, tight) and both isotopy types (topological and rigid vertex) extend naturally to bonded knotoids, providing a consistent framework at the level of open curves.

In Section~\ref{sec6:bondedbraids}, we turn to the algebraic counterpart of bonded links: \emph{bonded braids}. We define bonded braids as braids with bonds connecting strands, and we introduce the \emph{bonded braid monoid} $BB_n$ on $n$ strands, giving a complete presentation by generators and relations and two reduced ones. Moreover, we establish its isomorphism to the singular braid monoid. In Section ~\ref{alsec} we prove the analogue of the Alexander theorem in the topological bonded setting: every topological bonded link can be obtained as the closure of a bonded braid. 

In Section~\ref{mtbb}, we prove bonded braid equivalence theorems, showing that two bonded braids have topologically equivalent closures if and only if they are related by  certain moves adapted to bonds  and (tight) bonded braid isotopy. We begin with the adaptation of classical $L$-moves to (tight) bonded braids, which we further extend  to the more subtle (tight) bonded $L$-moves. These moves provide a more fundamental understanding of how bonded isotopies translate into more algebraic moves analogous to the classical Markov theorem.

In Section~\ref{fbb}, we introduce \emph{enhanced bonded links and braids}. An \emph{enhanced bond} comes in two types, representing, for example, an attractive vs. a repelling interaction. Mathematically, an attracting bond and a repelling bond may be thought of as mutual inverses. We show how allowing two bond types effectively turns the bonded braid monoid into a \emph{bonded braid group}, and we define this enhanced bonded braid group $EB_n$ with its extended generator set. The analogues of the Alexander and the Markov theorems are established for the enhanced setting as well.

Finally, in Section~\ref{sectl} we study \emph{bonded knotoids}, which are knotoid diagrams equipped with bonds. Bonded knotoids are especially relevant for modeling open chains such as proteins, since they allow one to represent both the open ends and internal bonds in a single diagram. We give the formal definition of bonded knotoids and discuss their equivalence moves (taking care to include the usual knotoid forbidden moves at endpoints as well as new forbidden moves involving bonds). 

We then define closure operations for bonded knotoids: connecting the endpoints with an underpassing or overpassing arc produces a bonded knot or link, analogous to Turaev's knotoid closure operations. We note that different closure choices can lead to different knots, so a fixed closure convention is assumed when using knotoids to represent specific bonded knots. We also consider a new \emph{semi-closure} operation in the enhanced context, wherein an attracting bond directly connects the two endpoints of a knotoid (indicating a physical tendency for the ends to come together). 

Finally, we introduce \emph{bonded braidoids}, the algebraic counterparts of bonded knotoids, which extend braid theory to open strands equipped with bonds and provide a natural setting for describing the algebraic structure of bonded knotoids and their closures.

We conclude with a discussion of further directions, including the further algebraic exploration of the bonded braid equivalences, the  extension of the study to the plat closure of bonded braids, the formulation of a bonded Morse category and the extension of the braidoid–knotoid interaction to the bonded and enhanced bonded settings, which we plan to develop in future works.  

Much of this work has been presented in conferences in a number of places, such as University of Ljubljana, Banff International Research Station, SKCM$^2$ WPI Hiroshima University,  Vrije Universiteit Amsterdam, Odessa National University of Technology.

\section{Bonded Links}\label{sec2:bondedlinks}

Knotted objects with bonds and their isotopy moves  have been  introduced for the first time in the context of bonded knotoids  \cite{GGLDSK} and then as bonded knots with distinguishable bonds in \cite{G}. In this section we consider  bonded knots and links and their isotopy moves in full generality, in the sense that the bonds can be knotted and linked.

\subsection{Definition of Bonded Links}

An (oriented) {\it  link on $c$ components} is an embedding of  $c$  (oriented)  circles $S^1$ in the 3-sphere $S^3$. For $c=1$ a link on a single component is a {\it  knot}. Classical knots and links are considered up to {\it  isotopy} of the ambient space $S^3$, namely orientation-preserving homeomorphisms taking the one knot or link to the other. We usually study knots and links via their {\it diagrams} which are regular projections on the plane with over/under conventions at the double points, the {\it crossings}, and where isotopy is translated into planar isotopy (see left hand illustration of Figure~\ref{planar}) and the Reidemeister moves (see Figure~\ref{breid1}).

Informally, a \emph{bonded link} is a knot or link together with a set of auxiliary arcs (the bonds) connecting pairs of points on the  knot or link, which can be thought of as connecting chords. By (bonded) "knots" or "links" we shall be referring invariably to both (bonded) knots and links. Formally, we have: 

\begin{definition} \label{def:bonded}
A {\it bonded link} (or {\it bonded knot}, in case of a single component) is a pair $(L, B)$, where $L$ is a link in $S^3$, and $B$ is a set of $k$ disjoint  arcs properly embedded in the complement $S^3 \setminus L$, such that each bond arc in $B$ has its two endpoints, called {\it nodes}, attached  on $L$. The nodes are attached transversely to $L$, with each node attaching at a distinct point on $L$.  If $B=\emptyset$, then $(L,\emptyset)$ is just a classical link. In this general setting, where bonds can be knotted and linked,  they will be referred to as {\it long bonds}.

 A {\it bonded link diagram} is a planar projection of a bonded link $(L,B)$ with the usual  knot diagram over/under conventions at double points, the {\it crossings}, which may occur entirely between link arcs or bonded arcs, or may involve both a link arc and a bonded arc. A neighbourhood of a node, depicted as $\vdash$ or $\dashv$, shall be referred to as   {\it bonding site}. 

 An {\it oriented bonded link} is obtained by assigning orientations to the link components of $L$ (ignoring the bonds).
\end{definition} 

 In figures we  depict a bond as a slightly thick or colored arc (like orange), view left hand illustration of Figure~\ref{bondgr} for an example. So, a bond is an auxiliary structure that does not intersect $L$ except at its endpoints, and it may connect two points on the same component of $L$ or on different components. We may imagine $L$ as a closed polymer chain (like a protein backbone) and the bonds are additional connections (like disulfide bridges or other interactions) that bind together parts of the chain. 

\begin{figure}[H]
\begin{center}
\includegraphics[width=0.8\textwidth]{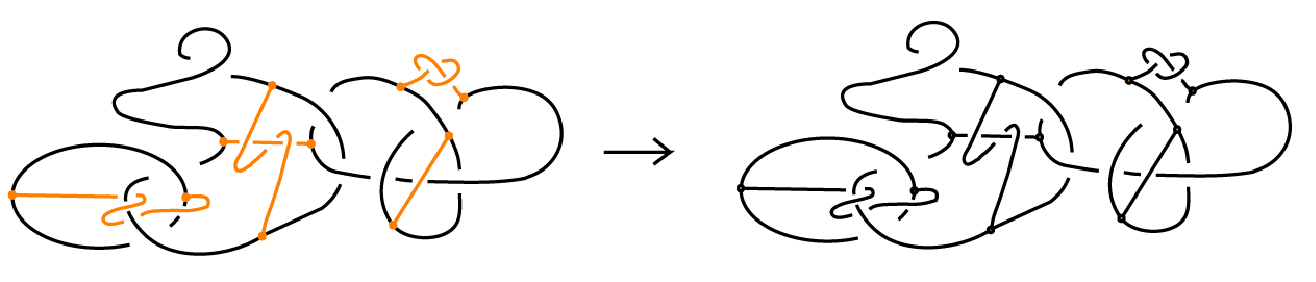}
\end{center}
\caption{On the left: a bonded link diagram, where the bonds (orange arcs) connect two points on the knot, with over/under-crossing information assigned at each bond crossing a link strand. On the right: the corresponding trivalent graph representation.}
\label{bondgr}
\end{figure}  

 \begin{remark}
Bonded knots/links can be viewed as special cases of embedded trivalent graphs \cite{K},  where there are two types of edges: {\it ordinary edges} and {\it bonds (orange)}. Bonds always start and terminate at a classical edge and the nodes  are the {\it vertices}.   
 \end{remark}


\subsection{Bonded Isotopy and Allowed Moves}

Intuitively, two bonded links $(L, B)$ and $(L', B')$ in \( S^3 \) are equivalent if one can be continuously deformed into the other via an ambient isotopy of \( S^3 \) that carries \( L \) to \( L' \) and \( B \) to \( B' \), i.e., an ambient isotopy that respects the bonds.  However, in the theory of bonded links we distinguish two types of equivalence relations, which we call \emph{topological vertex isotopy} and \emph{rigid vertex bonded isotopy}.

\begin{definition}
Two (oriented) bonded links $(L_1, B_1)$ and $(L_2, B_2)$ are {\it equivalent} as bonded links via {\it topological vertex isotopy} if there is an orientation-preserving homeomorphism $f: S^3 \to S^3$ taking $L_1$ to $L_2$ and $B_1$ to $B_2$, such that $f$ respects the attachment of bonds to link arcs (i.e. $f$ carries each bond in $B_1$ to a bond in $B_2$ connecting a corresponding pair of points on $L_2$). 
 In analogy, the (oriented) bonded links $(L_1, B_1)$ and $(L_2, B_2)$ are {\it equivalent} via {\it rigid vertex isotopy} if the nodes are considered to lie in local discs which are preserved by the isotopy.  Bonded links subjected to topological vertex isotopy shall be called {\it  topological bonded links}, while bonded links subjected to rigid vertex isotopy shall be called {\it  rigid bonded links}. 
\end{definition}

We shall further define the notion of trivial bonded link, which in knot theory corresponds to the notion of unknot and unlinks. 

\begin{definition}\label{def:trivial} 
A bonded link is {\it trivial} as topological resp. rigid bonded link if it can be turned by topological vertex isotopy resp. rigid vertex isotopy  into a planar graph.
\end{definition}

It is easy to create examples of bonded knots which are trivial as topological but non-trivial as rigid bonded knots. Clearly, rigid vertex isotopy implies topological vertex isotopy.  On the level of diagrams, this means that one can go from a diagram of $(L_1,B_1)$ to a diagram of $(L_2,B_2)$ by a finite sequence of planar isotopies and a set of local Reidemeister  moves involving any types of arcs (link arcs or  bonds or both) respecting or not rigid vertex isotopies. To see this, we adapt \cite[Theorem~2.1]{K} and~\cite[Section III]{K} to the bonded link setting and we obtain the following:

\begin{proposition}\label{breidthm}
Two (oriented) bonded links $(L_1, B_1)$ and $(L_2, B_2)$ are {\it equivalent} via topological vertex isotopy in $S^3$ if and only if any corresponding diagrams of theirs differ by a finite sequence of the following basic moves:
\begin{itemize}
    \item[i.] Planar isotopies of link arcs, bonded arcs, or nodes, as shown in Figure~\ref{planar}.
    \item[ii.] Classical Reidemeister moves (R1, R2, R3) acting on link arcs away from nodes, as exemplified in Figure~\ref{breid1}.
    \item[iii.] Reidemeister-type moves acting only on bonded arcs, as shown in Figure~\ref{brarcs}.
    \item[iv.] Reidemeister-type moves involving both link arcs and bonded arcs, as exemplified in Figure~\ref{mixed}.
    \item[v.] Vertex slide moves, allowing a bond endpoint to slide along the link to a new position, exemplified in Figure~\ref{breid2}.
    \item[vi.] Topological vertex twist moves (TVT), as exemplified in Figure~\ref{tvt}.
\end{itemize}
Similarly, two (oriented) bonded links $(L_1, B_1)$ and $(L_2, B_2)$ are {\it equivalent} as bonded links via {\it rigid vertex isotopy} if corresponding diagrams of theirs differ by a finite sequence of the moves (i)-(v) above together with the rigid vertex twists (RVT) shown in Figure~\ref{rvt}, excluding the TVT moves (Figure~\ref{tvt}). Rigid vertex isotopy moves are illustrated in Figures~\ref{planar}, \ref{breid1}, \ref{brarcs}, \ref{mixed}, \ref{breid2}, and \ref{rvt}.
\end{proposition}

We now present explicitly the moves listed in Proposition~\ref{breidthm}.
 Planar isotopy moves involving both link arcs and bonds are shown in Figure~\ref{planar}. Planar isotopy includes also sliding a node along its attaching arc. 

\begin{figure}[H]
\begin{center}
\includegraphics[width=4.8in]{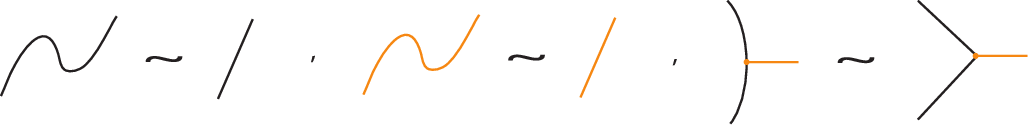}
\end{center}
\caption{The planar isotopy moves.}
\label{planar}
\end{figure}

We have all the classical Reidemeister moves (R1, R2, R3) acting on the link arcs of $L$ (away from any nodes) (as in  Figure~\ref{breid1} and their variants). 

\begin{figure}[H]
\begin{center}
\includegraphics[width=3.9in]{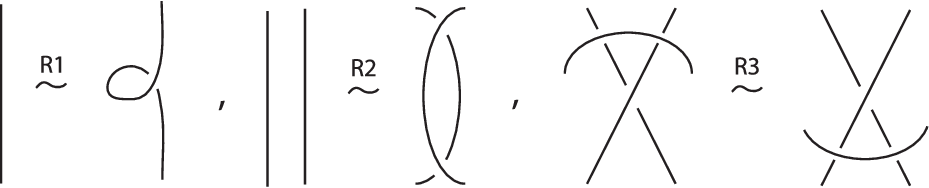}
\end{center}
\caption{ Reidemeister moves on link arcs.}
\label{breid1}
\end{figure}

In addition, we have the moves R1, R2, R3 that only involve bonded arcs (as in  Figure~\ref{brarcs} and their variants). 

\begin{figure}[H]
\begin{center}
\includegraphics[width=3.9in]{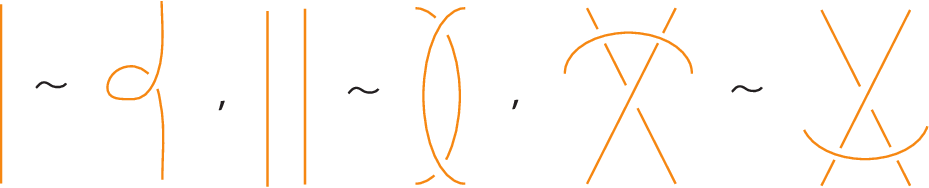}
\end{center}
\caption{Reidemeister moves on bonded arcs.}
\label{brarcs}
\end{figure}

We also allow Reidemeister moves between link and bonded arcs (see Figure~\ref{mixed}).

\begin{figure}[H]
\begin{center}
\includegraphics[width=4.8in]{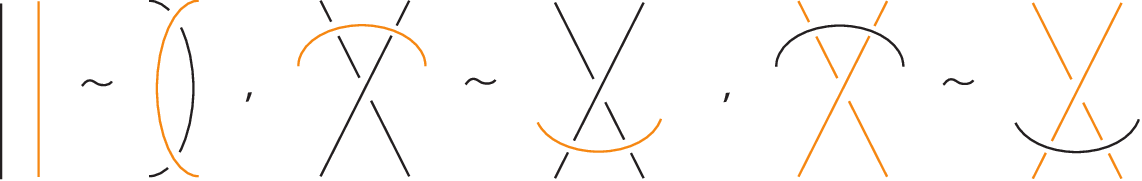}
\end{center}
\caption{Reidemeister moves between link and bonded arcs.}
\label{mixed}
\end{figure}

Moreover, a bond can pass through an arc of the link or through another bond (or vice versa), analogous to the R2 move as illustrated in Figure~\ref{breid2}. We call these moves {\it vertex slide moves}.

\begin{figure}[H]
\begin{center}
\includegraphics[width=3.6in]{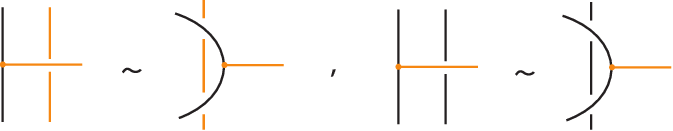}
\end{center}
\caption{Vertex slide (VS) moves.}
\label{breid2}
\end{figure}

In the topological vertex isotopy we  allow moves analogous to the classical R1 moves involving link arcs and bonds as illustrated in Figure~\ref{tvt}. We call these moves {\it topological vertex twists (TVT)}. 

\begin{figure}[H]
\begin{center}
\includegraphics[width=2.7in]{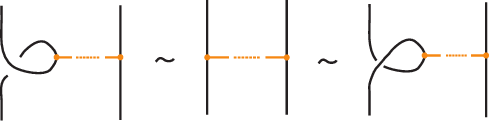}
\end{center}
\caption{Topological vertex twists (TVT).}
\label{tvt}
\end{figure}

Finally, Figure~\ref{rvt} illustrates the moves  rigid vertex twists (RVT) which replace the TVT moves for the rigid vertex isotopy setting, obtaining thus a stricter diagrammatic equivalence relation. 

\begin{figure}[H]
\begin{center}
\includegraphics[width=3.3in]{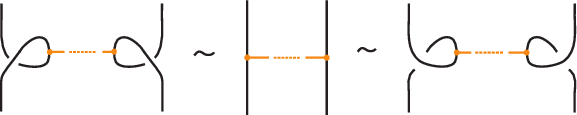}
\end{center}
\caption{Rigid vertex twists (RVT).}
\label{rvt}
\end{figure}

\noindent {\it Proof of Proposition~\ref{breidthm}.} This is an adaptation of~\cite[Theorem~2.1]{K} to the case of trivalent topological graphs in the particular form of bonded links. The only differences are as follows. Kauffman's original statement of the basic moves of topological vertex isotopy on embedded graphs includes the following moves (adapted to the bonded category) in place of TVT and VS, see Figure~\ref{rotslidem}:

\begin{figure}[H]
\begin{center}
\includegraphics[width=3.5in]{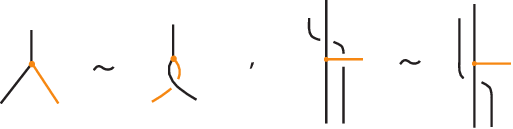}
\end{center}
\caption{A mixed topological vertex twist, mTVT, and the bonded slide moves, BS.}
\label{rotslidem}
\end{figure}

These moves, however, follow easily from the moves in Figures~\ref{breid1}, \ref{breid2}, and~\ref{tvt}, as explicitly demonstrated in Figures~\ref{proof_slide} and \ref{rl}.

\begin{figure}[H]
\begin{center}
\includegraphics[width=2.5in]{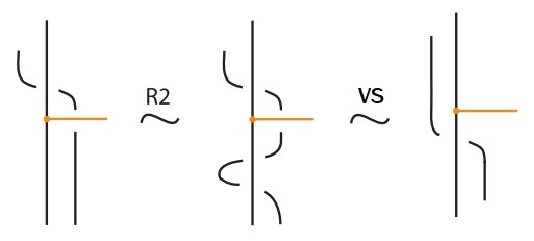}
\end{center}
\caption{A bonded slide move follows from the R2 and VS moves. So it is valid in the topological and rigid vertex isotopy.}
\label{proof_slide}
\end{figure}

\begin{figure}[H]
\begin{center}
\includegraphics[width=4.9in]{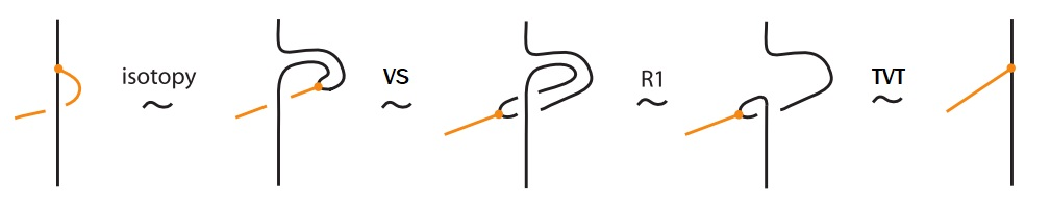}
\end{center}
\caption{An mTVT follows from  the VS, R1  and the  TVT moves. So it is valid (only) in the topological vertex isotopy.}
\label{rl}
\end{figure} 

The case of rigid vertex isotopy for bonded links is an adaptation of~\cite[Section III]{K} to the case of trivalent rigid vertex graphs in the particular form of bonded links and follows analogously. Hence the proof is complete.
\hfill $\Box$

\subsection{More Diagrammatic Moves for Bonded Links} 

We now present additional local diagrammatic moves in the setting of topological bonded links, and show that they all follow from the list of basic moves stated in Proposition~\ref{breidthm} and exemplified in Figures~\ref{planar}--\ref{tvt}. 

These additional moves fall into two main categories:  
\begin{itemize}
    \item moves involving the interaction of a bonded arc with a link arc or another bond,  
    \item moves involving the interaction of a bonded arc with a crossing.
\end{itemize}

We analyze these cases below.  

\medskip

\noindent\textbf{Bond–link arc interaction.}  
When a bonded arc interacts with a link arc to which it is attached, we have the so-called \emph{min/max sliding moves} (see Figure~\ref{minmaxproof}). These moves involve sliding a bonded arc along a link arc to which it is attached, creating or removing a crossing between the bond and the link arc. In Figure~\ref{minmaxproof} we show that these moves follow from the basic topological moves listed in Proposition~\ref{breidthm}.

\begin{figure}[H]
\begin{center}
\includegraphics[width=5.3in]{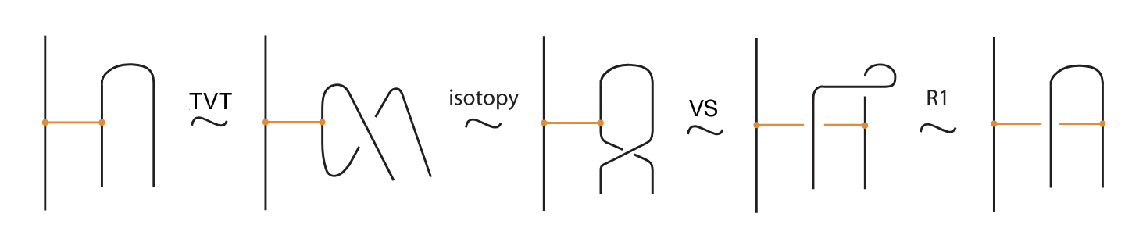}
\end{center}
\caption{The min/max sliding moves.}
\label{minmaxproof}
\end{figure}

In addition to the min/max sliding moves we also have the \emph{arc slide moves}.
An arc can also slide across a bond: entirely within the region between the nodes (\emph{internal arc slide}) (see Figure~\ref{slm}), outside the region (\emph{external arc slide}) (see Figure~\ref{br3}), or partially within (\emph{bonded slide}) (recall Figure~\ref{proof_slide}).

\begin{figure}[H]
\begin{center}
\includegraphics[width=2.1in]{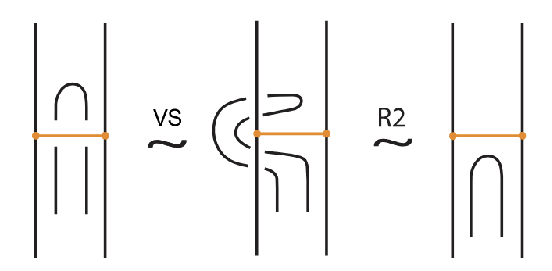}
\end{center}
\caption{An internal arc slide move expressed in terms of the basic moves.}
\label{slm}
\end{figure}

\begin{figure}[H]
\begin{center}
\includegraphics[width=4.8in]{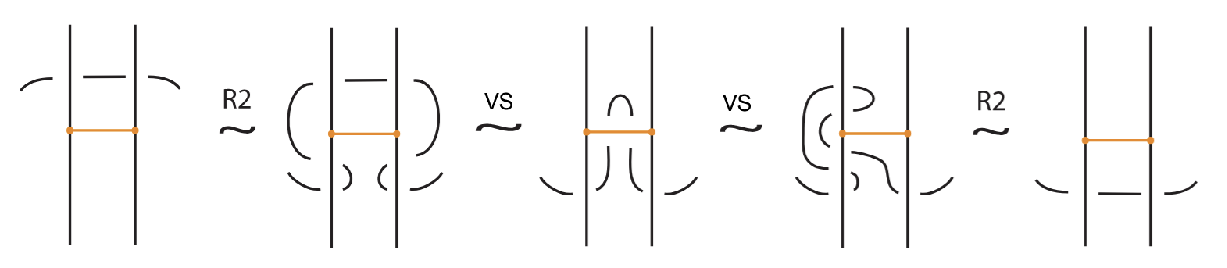}
\end{center}
\caption{An external arc slide move expressed in terms of the basic moves.}
\label{br3}
\end{figure}

\medskip

\noindent\textbf{Bond–crossing interaction.}  
Next, we consider moves involving the interaction of a bond with a crossing.  
Three scenarios arise: (i) both arcs of the crossing lie within the region of a bond, (ii) one arc provides a bonding site, (iii) both arcs provide a bonding site.  
The first two comprise the \emph{crossing slide moves} (Figures~\ref{proof+crslm} and \ref{proof+crslm1}); the latter gives rise to the \emph{bonded flype} and \emph{bonded double flype} moves (Figure~\ref{prooflem1}).
In Figures~\ref{proof+crslm}, \ref{proof+crslm1}, \ref{prooflem} and \ref{prooflem1} we show that these moves follow from the basic topological moves listed in Proposition~\ref{breidthm}.

\begin{figure}[H]
\begin{center}
\includegraphics[width=5in]{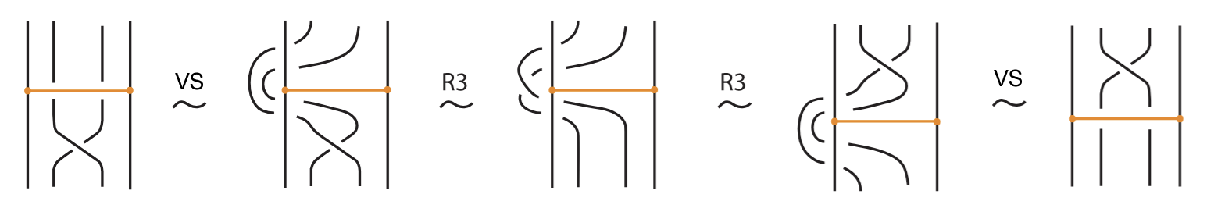}
\end{center}
\caption{The interior crossing slide moves follow from the basic topological moves.}
\label{proof+crslm}
\end{figure}

\begin{figure}[H]
\begin{center}
\includegraphics[width=2.8in]{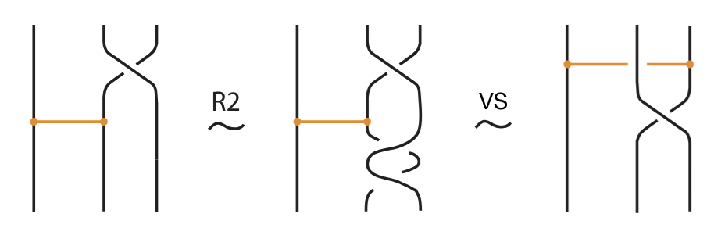}
\end{center}
\caption{A crossing slide move expressed in terms of the basic moves.}
\label{proof+crslm1}
\end{figure}

The bonded flype move follows from the TVT move (manifested as bond rotation,  that is, mTVT moves) and crossing slide moves, as demonstrated in Figure~\ref{prooflem}. The bonded double flype move is obtained by combining a bonded flype with an R2 move. The bonded double flype move, so also the bonded flype move, follows easily also in the rigid vertex isotopy  by the basic moves.

\begin{figure}[H]
\begin{center}
\includegraphics[width=5in]{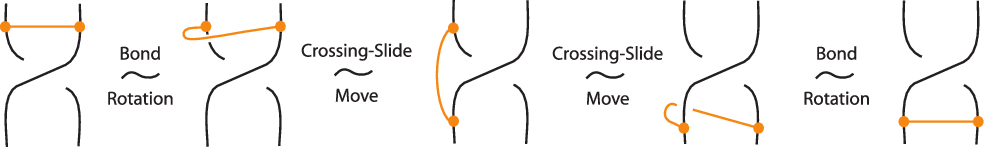}
\end{center}
\caption{The bonded flype move follows from the bond rotation and the crossing slide move.}
\label{prooflem}
\end{figure}
\begin{figure}[H]
\begin{center}
\includegraphics[width=2.1in]{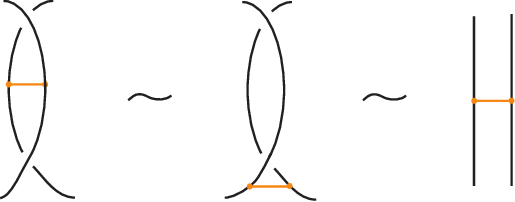}
\end{center}
\caption{The bonded double flype move follows from the bonded flype and an R2 move.}
\label{prooflem1}
\end{figure}


\subsection{Some Restrictions and Forbidden Moves}

There are some important restrictions (forbidden moves) in bonded link theory that do \emph{not} appear in classical knot theory. Notably, because bonds are \emph{embedded arcs}, one cannot slide a bond around a strand in such a way that the bond passes \emph{through two consecutive crossings of different type} (over then under, or vice versa). In effect, a bond cannot be pulled through a zig-zag in the link that would cause its endpoint to swap which side of the link it is on. Such a move would entail the bond endpoint going through a crossing, which in the topological model is not allowed unless it follows the allowed moves described above. Figure~\ref{for1} illustrates these forbidden moves. 

\begin{figure}[H]
\begin{center}
\includegraphics[width=3.6in]{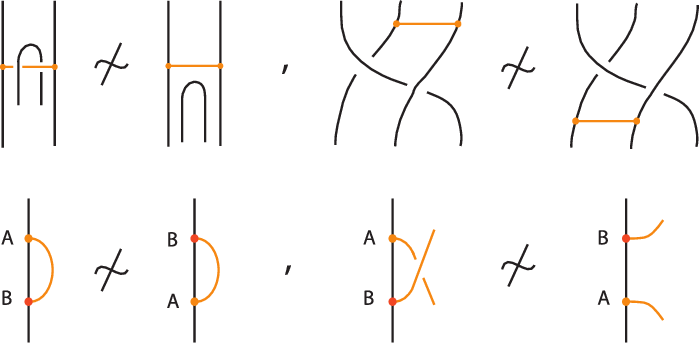}
\end{center}
\caption{Some forbidden moves in the theory of bonded links.}
\label{for1}
\end{figure}

Another related diagrammatic setting is the theory of tied links \cite{AJ1,D2}, links equipped with ties. A \emph{tie} is a non-embedded simple arc  connecting different points of a link, whose ends can slide along the arcs they are attached but remain on the link. Ties are like phantom bonds: they can pass freely through one another and through the arcs of the link. Two ties can even merge into one if they connect the same arcs. Bonded links differ crucially: bonds are embedded arcs and cannot be created or destroyed by isotopy, nor can they pass through link arcs or through themselves arbitrarily. 

\section{Standard Bonded Links}\label{sec3:stbondedlinks}

Given a bonded link diagram, it is  possible to \emph{simplify} the diagram by applying isotopy moves so that the bonds themselves become unknotted and are presented in a ``standard'' form. Knotted objects with bonds in standard form and their isotopy moves  have been first introduced in the theory of bonded knotoids \cite{GGLDSK} and later as bonded knots with distinguishable bonds in \cite{G}.

Recall that in a bonded link diagram, bonds may initially be knotted or entangled on their own, forming non-trivial configurations in space. However, by performing appropriate isotopies, one can eliminate such knotting in the bonds and arrange them in a simple configuration where each bond is unknotted. In particular, it is useful to isotope the diagram so that each bond connects two link segments in a simple \emph{H-shaped} configuration. See Figure~\ref{regbli} for an example. Consider a small neighborhood around each node on the link: the diagram can be isotoped so that near each bond endpoint the link segment appears roughly vertical and the bond itself attaches like a horizontal rung, together forming a shape reminiscent of the letter ``H'', with the bond as the horizontal bar and the link segments as the vertical bars. We refer to such a local configuration as an \emph{H-neighborhood of a bond} (see Figure~\ref{HN}).

\begin{definition}\rm
A {\it standard bonded link diagram} is a bonded link  diagram  that contains no crossings or self-crossings of bonded arcs.
\end{definition}

\begin{figure}[H]
\centering
\begin{minipage}{0.5\textwidth}
  \centering
\includegraphics[width=2in]{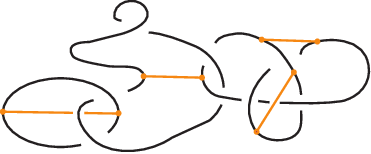}
\caption{A standard bonded link.}
\label{regbli}
\end{minipage}%
\hfill
\begin{minipage}{0.45\textwidth}
  \centering
\includegraphics[width=.8in]{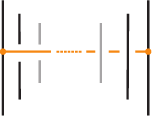}
 \caption{An H-neighborhood of a bond (standard form). The bond (horizontal) connects two link arcs (vertical).}
\label{HN}
\end{minipage}
\end{figure}

We now show that any bonded link diagram can be transformed into a standard form by a sequence of bonded isotopy moves: 

\begin{proposition}
A (topological or rigid) bonded link diagram can be transformed isotopically into a standard bonded link diagram.  
\end{proposition}

\begin{proof}
Consider a bond that is not in standard form, and let $J_1$ and $J_2$ denote its two nodes. Let $D$ be a disc containing only $J_1$ and the three arcs emanating from it, one belonging to the bond and the other two to the link. Using the  vertex slide moves (recall Figure~\ref{breid2}), we can slide $D$ along the bond toward $J_2$, progressively eliminating any self-crossings of bonded arcs. During this process, new crossings may appear, but only between link arcs, and never between arcs of the bonds themselves (see Figure~\ref{examplereg}). The procedure is then completed by induction on the number of bonds initially not in standard form, successively applying this process to each bond until all are brought into standard form.
\end{proof}

\begin{figure}[H]
\begin{center}
\includegraphics[width=5.1in]{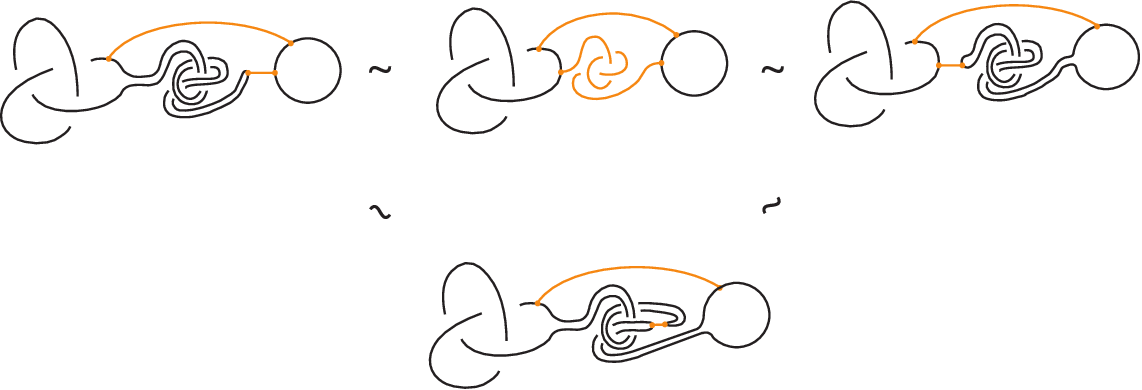}
\end{center}
\caption{An example of bond contraction.}
\label{examplereg}
\end{figure}

\begin{proposition}
Given a long bond between two nodes as exemplified in Figure~\ref{examplereg}, we can keep one node fixed and contract the bond into a neighborhood of that node. This gives rise to two contractions, one for each node. Call these contracted diagrams $D_1$ and $D_2$. Then $D_1$ and $D_2$ are topologically isotopic or rigid vertex isotopic (depending on context) as standard bonded diagrams. 
 \end{proposition}
\begin{proof}
Indeed, the long arc of the original bond provides a track along which to move the contracted bond from the one side to the other, as shown in  Figure~\ref{examplereg}.  
\end{proof}

We now state the following theorem describing topological and rigid   equivalence diagrammatically for standard bonded links. 

\begin{theorem} \label{reid_standard}
Two (oriented) standard bonded links are topologically isotopic if and only if any corresponding diagrams of theirs differ by a finite sequence of the moves illustrated in Figures~\ref{planar}, \ref{breid1} and \ref{rigidiso_top} with all variants.

Similarly, two (oriented) standard bonded links are rigid vertex isotopic if and only if any corresponding diagrams of theirs differ by a finite sequence of the moves illustrated in Figures~\ref{planar}, \ref{breid1} and \ref{rigidiso_reg} with all variants.  
\end{theorem}

\begin{figure}[H]
\begin{center}
\includegraphics[width=4.5in]{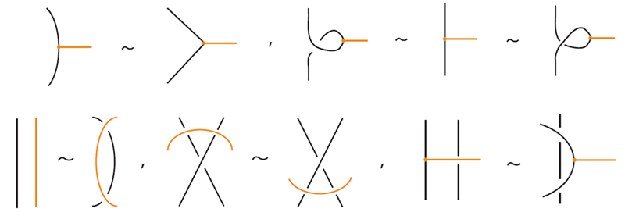} 
\end{center} 
\caption{Topological  vertex isotopy moves between standard bonded links.} 
\label{rigidiso_top}
\end{figure}

\begin{figure}[H]
\begin{center}
\includegraphics[width=4.5in]{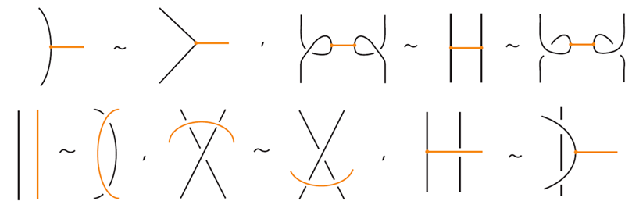} 
\end{center} 
\caption{Rigid vertex isotopy moves between standard bonded links.} 
\label{rigidiso_reg}
\end{figure}

\section{Tight Bonded Links}\label{sec4:tightbondedlinks}

We now consider a stricter category of bonded links, called \emph{tight bonded links}. While standard bonded links allow bonds to cross over or under link strands (provided the bonds themselves remain unknotted), tight bonded links forbid such crossings entirely. See Figure~\ref{tightbl} for an exmaple. In a tight bonded link, each bond sits in an H-neighborhood that is free of arcs, as illustrated in Figure~\ref{hnei}. We have the following definition (compare with \cite{GGLDSK} and \cite{G}):

\begin{definition} \label{def:tight}
A \emph{tight bonded link diagram} is a bonded link diagram in which each bond is presented in standard form and, additionally, no bond crosses any link arc. 
\end{definition}

\begin{figure}[H]
\centering
\begin{minipage}{0.4\textwidth}
  \centering
  \includegraphics[width=2in]{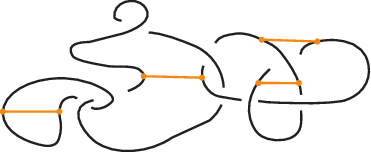}
  \caption{A tight bonded link.}
  \label{tightbl}
\end{minipage}%
\hfill
\begin{minipage}{0.4\textwidth}
  \centering
  \includegraphics[width=0.3in]{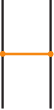}
  \caption{The H-neighborhood of a tight bond.}
  \label{hnei}
\end{minipage}
\end{figure}

\begin{proposition}
A standard bonded link diagram can be transformed isotopically (topologically and rigidly)  into a tight bonded link.  
\end{proposition}

\begin{proof}
Starting from a standard bonded link diagram, we apply the isotopy illustrated in Figure~\ref{regulartotight} to each bond individually. This isotopy removes any crossings between the bond and the link in the region between its two nodes, while preserving the standard H-neighborhood configuration of the bond. Repeating this process for all bonds in the diagram yields a tight bonded link diagram.
\end{proof}
\begin{figure}[H]
\begin{center}
\includegraphics[width=4in]{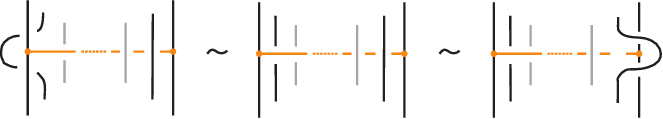}
\end{center}
\caption{A standard bond can be brought to tight form by  vertex slide moves.}
\label{regulartotight}
\end{figure}


We now state the following theorem describing topological and rigid   equivalence in the tight category through diagrammatic moves.

\begin{theorem} \label{tightequiv}
Two (oriented) tight bonded links are topological vertex isotopic in the tight category if and only if any corresponding diagrams of theirs differ by a finite sequence of planar isotopies and the Reidemeister moves for the link arcs (recall Figures~\ref{planar}, \ref{breid1}), and the moves illustrated in Figure~\ref{tightiso_top}. 
 Similarly, two (oriented) tight bonded links are rigid vertex isotopic in the tight category if and only if any corresponding diagrams of theirs differ by a finite sequence of planar isotopies and the Reidemeister moves for the link arcs, and the moves illustrated in Figure~\ref{tightiso_rigid}. 
\end{theorem}

\begin{figure}[H]
\begin{center}
\includegraphics[width=3.5in]{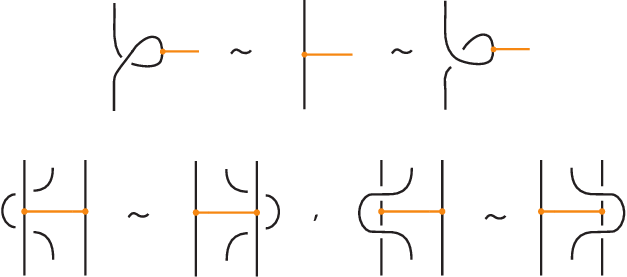}
\end{center}
\caption{Basic topological isotopy moves between tight bonded links.}
\label{tightiso_top}
\end{figure}

\begin{figure}[H]
\begin{center}
\includegraphics[width=3.3in]{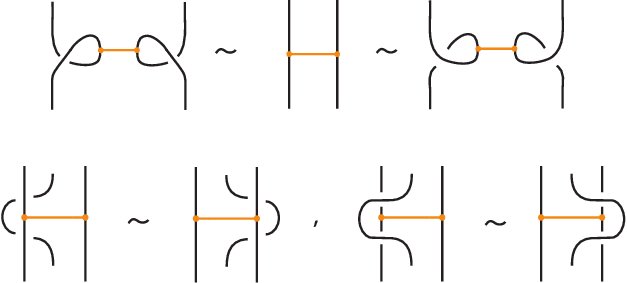}
\end{center}
\caption{Basic rigid vertex isotopy moves between tight bonded links.}
\label{tightiso_rigid}
\end{figure}

\begin{remark}
The move of Figure~\ref{br3} which may be seemingly needed to be included in the tight category, follows eventually from the other moves in the tight category as we demonstrate in Figure~\ref{R3iso}.
\end{remark}

\begin{figure}[H]
\begin{center}
\includegraphics[width=3.3in]{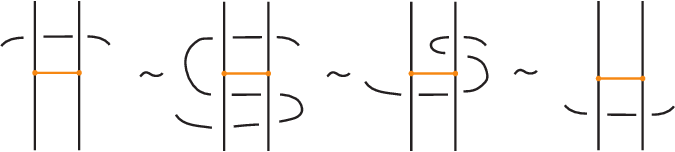}
\end{center}
\caption{An external arc slide move as a combination of basic moves in the tight category.}
\label{R3iso}
\end{figure}

\section{Invariants of Topological and Rigid Vertex Bonded Links}\label{inv}

In this section we briefly describe two general techniques for obtaining invariants for bonded links, corresponding to the two types of isotopy: topological vertex isotopy and rigid vertex isotopy, which are both  due to Kauffman, adapted to the category of bonded links. We further construct the bonded  bracket polynomial.

\subsection{Topological Bonded Links:  Unplugging Operations}

For a bonded link $(K,B)$, we view the link together with the bonds as a trivalent embedded graph. At each node (vertex) where a bond attaches, we may perform an \emph{unplugging} operation: removing one of the edges incident to the vertex as shown in the top row of Figure~\ref{unplugging}. At a 3-valent vertex there are three such choices. This yields a collection of  unknotted (redundant) arcs and classical links corresponding to the various choices of unpluggings. 
 A classical result due to Kauffman \cite{K} asserts, in the language of bonded links:

\begin{theorem} \label{th:unplugging}
Let $(K,B)$ be a bonded link. Then, considering the bonds and the link arcs as graph edges invariantly, we obtain:
\begin{itemize}
    \item The set $\mathcal{L}_{(K,B)}$ of classical links derived from all possible vertex unpluggings is a topological vertex isotopy invariant of $(K,B)$.
    \item Consequently, any ambient isotopy invariant applied to $\mathcal{L}_{(K,B)}$ is also a topological vertex isotopy invariant of $(K,B)$.
\end{itemize}
\end{theorem}

 Figure~\ref{unplugging} bottom row shows the set of three knots derived by the unplugging of the given bonded knot, seen as an embedded graph.

\begin{figure}[H]
\begin{center}
\includegraphics[width=5.2in]{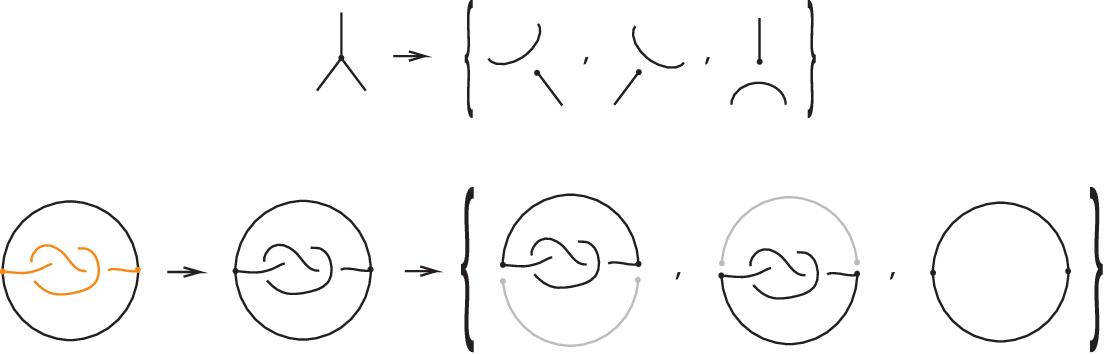}
\end{center}
\caption{Unplugging edges from an embedded graph.}
\label{unplugging}
\end{figure}

We note that we can create finer unplugging invariants for bonded links by restricting unplugging in relation to the bonds. There are two types, the  bonded unplugging and the strict  unplugging, which are complementary:

\begin{definition}
The {\it bonded unplugging} of a bonded link $(K,B)$ is the unplugging where we unplug only the bonds, so we are  left with the {\it underlying link} $K$. 
\end{definition}

It now follows immediately from Theorem~\ref{th:unplugging} that:

\begin{proposition}
    The underlying link $K$ of a bonded link $(K,B)$ is a topological vertex isotopy invariant of $(K,B)$.
\end{proposition}

The bonded unplugging can distinguish bonded knots. In Figure~\ref{nonisobk} the bonded unplugging yields an unknot in the one case and a trefoil knot in the other case. Hence the two bonded knots are distinguished. In this example with the knotted bond we have an embedding of a $\Theta$-graph. This knotted graph is the Turaev closure of the knotoid depicted by the knotted bond as a diagram in the surface of the 2-dimensional sphere. The fact that we have proved that the $\Theta$-graph has a non-trivial embedding implies that this knotoid is non trivial.

\begin{figure}[H]
\begin{center}
\includegraphics[width=2in]{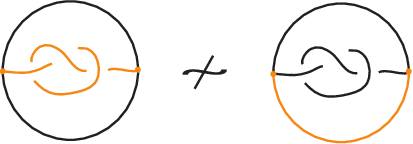}
\end{center}
\caption{The bonded unplugging distinguishes the two bonded knots.}
\label{nonisobk}
\end{figure}

\begin{definition}
The {\it strict unplugging} of a bonded link $(K,B)$ is the unplugging where at a given node only the arcs from the link are allowed to be unplugged.  
\end{definition}

At a 3-valent vertex there are two such choices. It also follows immediately from Theorem~\ref{th:unplugging} that:

\begin{proposition}
    The resulting collection of knots and links (ignoring the unknotted  arcs) from the strict unplugging of a bonded link  is a topological vertex isotopy invariant of the  bonded link.
\end{proposition}

Figure~\ref{unplugtypes} illustrates an example of a bonded knot $G$ which is trivialized by the bonded unplugging. However, via the strict unplugging we obtain an unknot and a trefoil knot. This proves that the given bonded knot is non-trivial. Also that the $\Theta$-graph where one considers the bond as another edge is non-trivial.  Let further $G_u$ be the same graph where now the upper arc in $G$ is a bond (see left illustration in Figure~\ref{unplugtypes}). If we perform bonded unplugging to $G_u$ we obtain a trefoil knot, therefore $G_u$ is not isotopic to $G$. 

\begin{figure}[H]
\begin{center}
\includegraphics[width=5in]{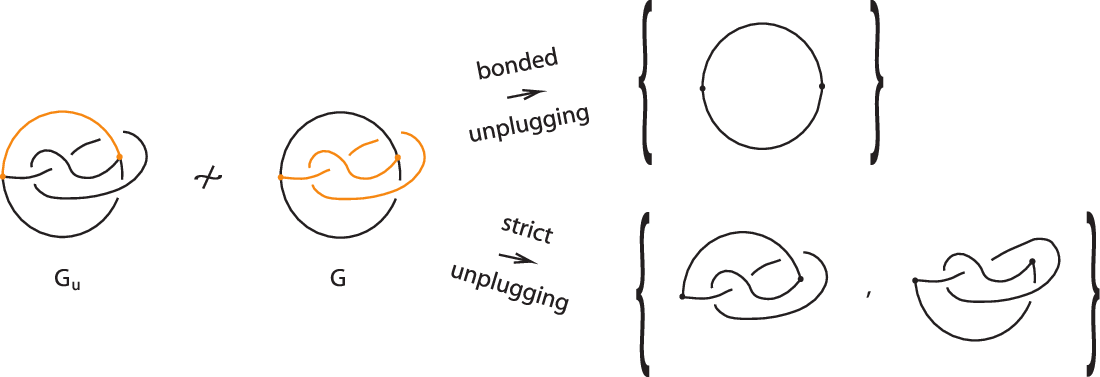}
\end{center}
\caption{Unplugging edges from an embedded graph.}
\label{unplugtypes}
\end{figure}

In the case of $G_l$ where the bond is the lower arc  of $G$, the unplugging invariants do not distinguish $G_l$ and $G$. In fact  $G_l$ and $G$ are isotopic as demonstrated in Figure~\ref{exisob}.

\begin{figure}[H]
\begin{center}
\includegraphics[width=4.3in]{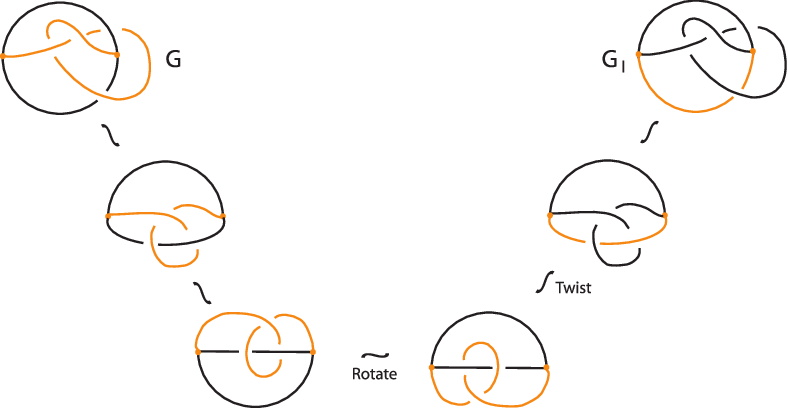}
\end{center}
\caption{Two isotopic bonded knots.}
\label{exisob}
\end{figure}


\subsection{Rigid Standard Bonded Links: Tangle Insertions}

Given a standard bonded link $(K,B)$, one may replace each bond by a properly embedded band. Then, at each band we  insert any classical 2–tangle. This process is illustrated in Figure~\ref{tangleinsert}. The resulting set of classical links for all possible choices of inserted 2–tangles, depends only on the rigid vertex isotopy class of $(K,B)$. 

\begin{figure}[H]
\begin{center}
\includegraphics[width=4.3in]{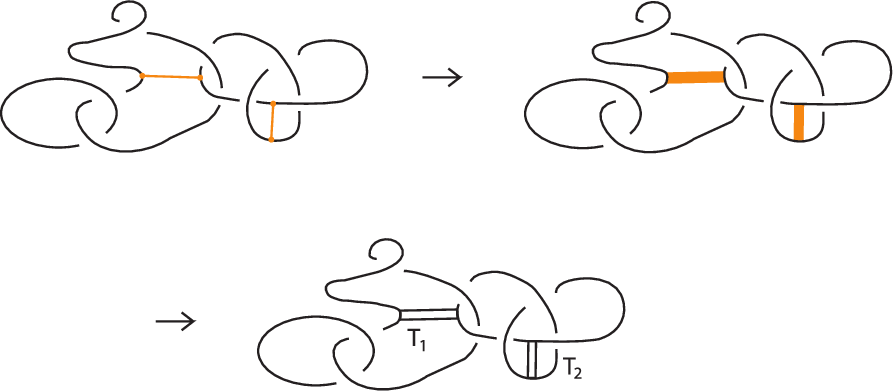}
\end{center}
\caption{Band replacements and tangle insertions.}
\label{tangleinsert}
\end{figure}

More precisely, given a rigid standard bonded link diagram $(L,B)$, we proceed as follows:
\begin{itemize}
    \item Replace each bond by a properly embedded band connecting its two nodes.
    \item Into each band, insert a properly embedded $2$–tangle chosen from a fixed family of tangles.
    \item The result is a classical link diagram whose isotopy class depends only on the rigid isotopy class of $(L,B)$ and the choice of tangles.
\end{itemize}
    
We now state the following result about tangle insertion, also due to Kauffman \cite{K} but adapted to the setting of rigid bonded links. Originally the tangle insertion would take place in the disc of a rigid vertex. Here the rigid vertex is replaced by a bond.  
The earliest example of tangle insertion is the replacement of a singular crossing by classical crossings in the theory of Vassiliev invariants.

\begin{theorem}\label{thm:tangins}
Let $(K,B)$ be a standard bonded link. Replace the bonds with properly embedded discs in the form of bands. Then:
\begin{itemize}
    \item The set $\mathcal{T}_{(K,B)}$ of classical links obtained from all possible tangle insertions is a rigid vertex isotopy (resp. rigid vertex regular isotopy) invariant of $(K,B)$. 
    \item For any fixed choice of tangles, the ambient (resp. regular) isotopy class of the resulting classical link $L_{(K,B)}$ is a rigid vertex  (resp. rigid vertex regular) isotopy invariant of $(K,B)$.
\end{itemize}
\end{theorem}

\begin{remark} \rm
Evaluating the resulting classical link diagram using any ambient or regular isotopy invariant (such as the Jones polynomial, the HOMFLY polynomial, the bracket polynomial, etc.) yields an invariant of the original rigid tight bonded link. 
Since the choice of tangles is arbitrary, and the choice of classical invariants is also arbitrary, this method generates infinitely many distinct invariants of rigid bonded links. In practice, one typically fixes a family of tangles and a preferred classical invariant to construct an invariant adapted to the application at hand.
\end{remark}

\begin{remark}
The tangle insertion technique is specific to the rigid vertex category, since in the topological category the nodes are allowed to rotate, making the inserted tangles ill-defined under isotopy. Thus, tangle insertion  is a natural tool for studying rigid bonded links.
\end{remark}

\subsection{Rigid Tight Bonded Links: The Bonded Bracket Polynomial}

In this subsection we present the \emph{bonded bracket polynomial}, which extends the Kauffman bracket polynomial to rigid tight  bonded links by assigning specific skein-theoretic relations at bonds. This invariant arises by choosing a particular $2$–tangle insertion at each bond, corresponding to a fixed linear combination of possible $2$–tangle states, weighted by formal coefficients. The resulting classical diagram (obtained after resolving all bonds via this prescribed $2$–tangle insertion) is then evaluated using the usual Kauffman bracket polynomial.

We emphasize that the construction described below uses one specific choice of $2$–tangle insertion at each bond. In principle, different choices of $2$–tangle insertions give rise to different invariants. Indeed, one can think of this invariant as a representative example of a much broader class of invariants defined by varying the $2$–tangle substitutions at bonds. Exploring how powerful this family of invariants is remains an open problem.

\begin{definition}\label{BKauff}\rm
Let $L$ be a tight bonded link diagram. The \emph{bonded bracket polynomial} of $L$, denoted $\langle L \rangle$, is defined inductively via the skein relations illustrated in Figure~\ref{bkauf}, where $a$ and $b$ are formal coefficients.
\end{definition}

\begin{figure}[H]
\begin{center}
\includegraphics[width=3.3in]{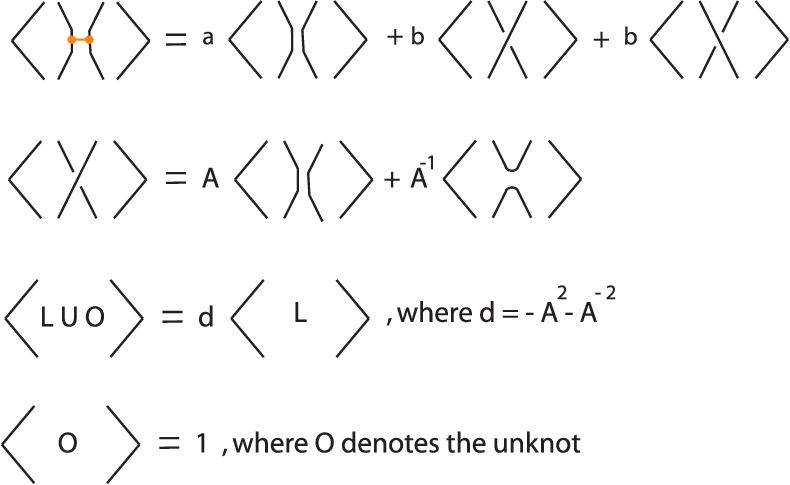}
\end{center}
\caption{Skein relations defining the bonded bracket polynomial.}
\label{bkauf}
\end{figure}

We now show that this polynomial is invariant under {\it regular isotopy}  (excluding move R1) for tight rigid bonded links.

\begin{proposition}\label{kaufprop}
The bonded bracket polynomial is an invariant of regular rigid vertex isotopy in the category of tight rigid bonded links.
\end{proposition}

\begin{proof} 
In the absence of bonds, the bonded bracket polynomial coincides with the classical Kauffman bracket polynomial, and thus is invariant under the classical Reidemeister moves R2 and R3. For the moves involving bonds, invariance under the local moves analogous to R2 at the nodes is verified directly using the skein relations, as illustrated in Figure~\ref{k2}.  

\begin{figure}[H]
\begin{center}
\includegraphics[width=4.3in]{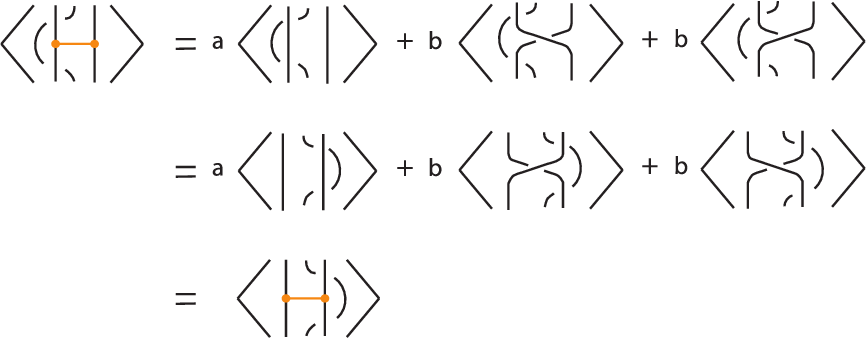}
\end{center}
\caption{Invariance under a vertex slide move.}
\label{k2}
\end{figure}

 For justifying rigid vertex  isotopy invariance we argue as follows. The bonded skein relation replaces each bond by a trivial 2-tangle or a positive or a negative type crossing. In the end we obtain a set of $3^k$ link diagrams for $k$ bonds. By Theorem~\ref{thm:tangins} this set of links and each one separately is a rigid vertex isotopy (resp. regular rigid vertex  isotopy)  invariant of the original  tight bonded link diagram $L$. Therefore, applying the Kauffman bracket polynomial to each one will retain the regular rigid vertex  isotopy invariance.  
\end{proof}

To illustrate the bonded bracket polynomial, we explicitly compute it for three  examples where, for simplicity, we set $b=1$.

\begin{example}\label{e1}\rm
Consider the unknot equipped with $n$ bonds, denoted $U_n$. 
\begin{figure}[H]
\begin{center}
\includegraphics[width=1.6in]{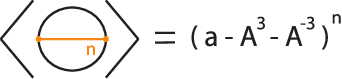}
\end{center}
\caption{The unknot with $n$ bonds.}
\label{ex1}
\end{figure}
Then:
\[
\langle U_n \rangle = \left(a - A^3 - A^{-3}\right)^n.
\]
The proof follows by induction, as illustrated in Figures~\ref{ex1} and~\ref{ex1b}, taking into account that the addition of a local curl to a bonded diagram by an R1 move multiplies the bracket polynomial of the diagram by a factor of $-A^3$ if the curl is right-handed and $-A^{-3}$ if the curl is left-handed. 
\end{example}

\begin{figure}[H]
\begin{center}
\includegraphics[width=4.5in]{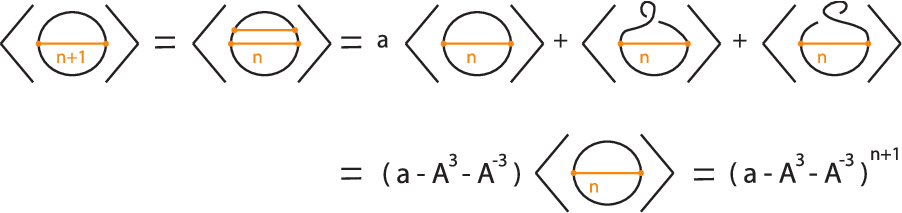}
\end{center}
\caption{Induction step for computing $\langle U_n \rangle$.}
\label{ex1b}
\end{figure}

\begin{example}\label{example1}\rm
Let $K_n$ denote the two–component unlink with $n$ bonds between the components (Figure~\ref{ex2}).
\begin{figure}[H]
\begin{center}
\includegraphics[width=.8in]{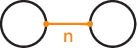}
\end{center}
\caption{The 2-component unlink with $n$-bonds between the two components.}
\label{ex2}
\end{figure}
Then we have the following:
\begin{equation}\label{unl}
\begin{array}{lcl}
\langle K_n \rangle & = & \left(a + A + A^{-1} \right)^{n-1} \left(ad -A^3 -A^{-3}\right)\, +\\
&&\\
&+&  \left(A+A^{-1}\right)\underset{i=1}{\overset{n-1}{\sum}}\, \left(a+A+A^{-1} \right)^{n-i-1} \,  \left(a-A^3-A^{-3}\right)^{i}
\end{array}
\end{equation}
\noindent where $d\, =\, -A^2-A^{-2}$. 
\smallbreak
We prove relations (\ref{unl}) by induction on $n\in \mathbb{N}\backslash \{0\}$. For $n=1$ we have that:
\begin{figure}[H]
\begin{center}
\includegraphics[width=3.7in]{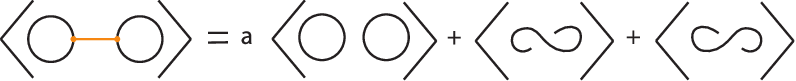}
\end{center}
\label{ex2b}
\end{figure}
\noindent Hence, $\langle K_1 \rangle =a\, d\, -\, A^3\, -\, A^{-3}$. Assume that relations (\ref{unl}) hold for $n$. Then, for $n+1$ we have:
\[
\begin{array}{ccl}
\langle K_{n+1} \rangle  & = & \left(a+A+A^{-1}\right)\, \langle K_n \rangle\, +\, (A+A^{-1})\, \langle U_n \rangle \ \overset{Ex.~\ref{e1}}{\underset{{\rm ind.\, step}}{=}}\\
&&\\
& = & \left(a + A + A^{-1} \right)^{n} \left(ad -A^3 -A^{-3}\right)\, +\, \left(A+A^{-1}\right)\left(a-A^3-A^{-3}\right)^n\, +\\
&&\\
& + & \left(A+A^{-1}\right)\underset{i=1}{\overset{n-1}{\sum}}\, \left(a+A+A^{-1} \right)^{n-i} \,  \left(a-A^3-A^{-3}\right)^{i}\, =\\
&&\\
&=& \left(a + A + A^{-1} \right)^{n} \left(ad -A^3 -A^{-3}\right)\, +\, \\
&&\\
& + & \left(A+A^{-1}\right)\underset{i=1}{\overset{n}{\sum}}\, \left(a+A+A^{-1} \right)^{n-i} \,  \left(a-A^3-A^{-3}\right)^{i}\\
\end{array}
\]
\end{example}

\begin{example}\label{example3}\rm 
In this example we consider in Figure~\ref{T1T2} two versions of bonded trefoil knots $T_1$ and $T_2$. In the figure the bonds are represented by curly black arcs.

\begin{figure}[H]
\begin{center}
\includegraphics[width=2.3in]{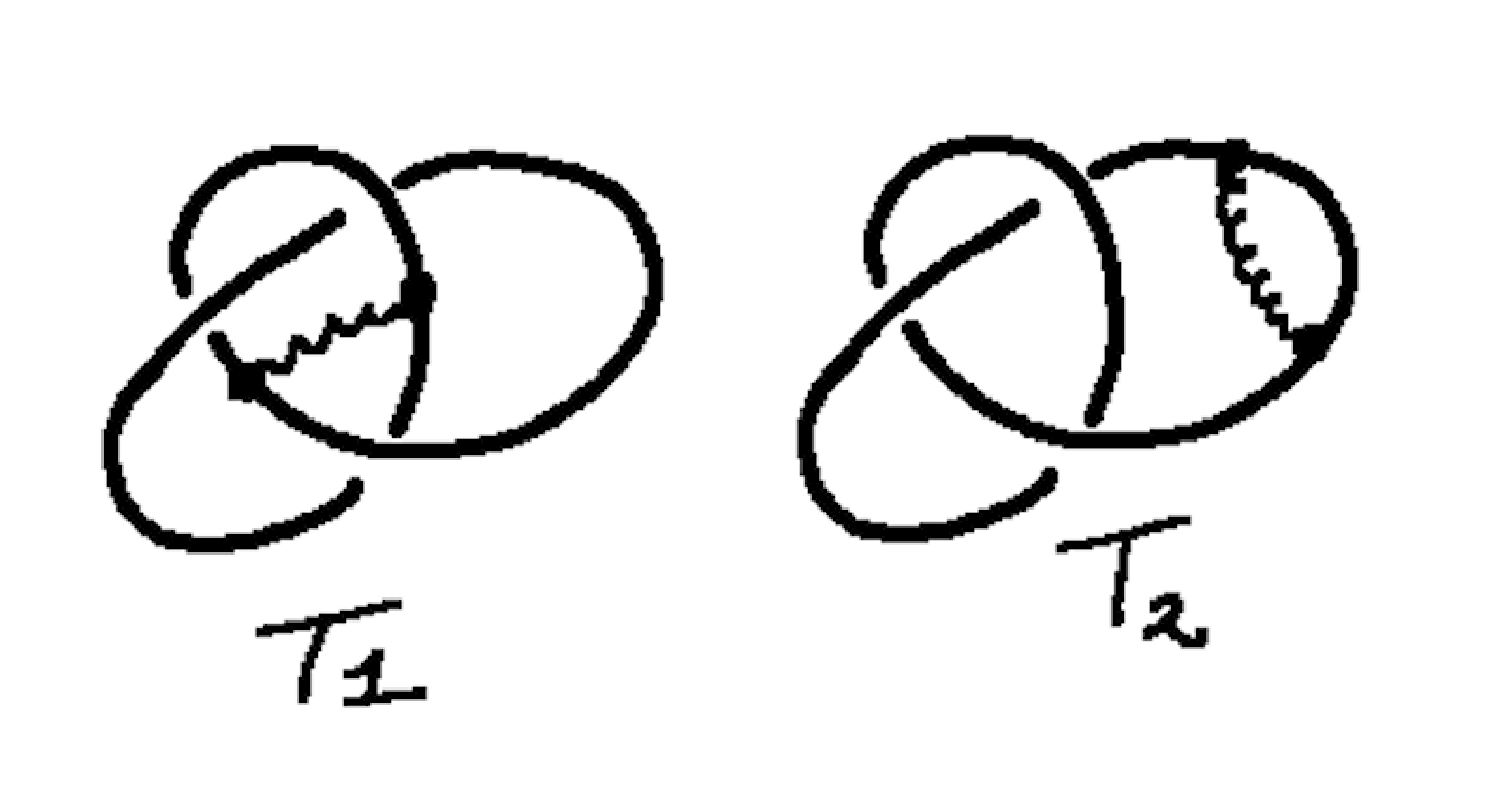}
\end{center}
\caption{Two bonded trefoils.}
\label{T1T2}
\end{figure}

By expanding the crossings in these examples via the Kauffman bracket identity, we find that
$$\langle T_1 \rangle = (A^{3} d + 3 A^{2} B + A B^{2} d) \langle \Theta_{1} \rangle + (2A B^{2} + B^{3} d) \langle D_{1} \rangle $$ 
and
$$ \langle T_2 \rangle =(A^{3} d + 3A^{2}B + AB^{2} + 2AB^{2} d + B^{3} d^{2}) \langle \Theta_{1} \rangle.$$
Here we use $B = A^{-1}$ and $d = -A^{2} - A^{-2}$. One can then see that $\langle T_1 \rangle$ and $\langle T_2 \rangle$ are sufficiently different to verify that $T_1$ and $T_2$ are distinct rigid vertex bonded knots. In fact, we specifically calculate that
$$\langle T_1 \rangle = 1 + 1/A^8 + a/A^7 + 1/A^6 - 1/A^4 - a/A^3 - A^4 - a A^5 + A^8$$ 
and that
$$\langle T_2 \rangle = -1/A^{10} + a/A^7 - 2/A^4 - 1/A^2 + a/A + a A - A^4 - a A^5 + A^{8}.$$
This shows explicitly that the two bonded knots are not rigid vertex isotopic.
\end{example}

We note that the highest power of $a$ appearing among all monomials in the expansion of the bonded bracket polynomial of a bonded link diagram equals the total number of bonds. 
Finally, we note that  invariants defined for standard bonded links extend to tight bonded links, and, in some cases, they become simpler to compute due to the absence of crossings between bonds and  link arcs.


\section{Bonded Braids}\label{sec6:bondedbraids}

In this section we develop the theory of \emph{bonded braids}, an algebraic framework that extends classical braids by including bonds. Note that we restrict our attention to the standard and tight categories of bonded links, undergoing topological vertex isotopy. Combined with the braiding result of next section, this algebraic structure could be used in encoding the topological structures of bonded links and of objects that they model.

We recall that a {\it geometric braid} on $n$ strands is a homeomorphic image of $n$ arcs in the interior of  $[0,1] \times [0,\epsilon] \times  [0,1]$, $\epsilon>0$, such that it is monotonous, that is, there are no local maxima or minima, and the boundary of the image consists in $n$ numbered points in $[0,1] \times [0,\epsilon] \times  \{0\}$ and $n$ corresponding numbered points in $[0,1] \times [0,\epsilon] \times \{1\}$. We study braids through their diagrams in the plane $[0,1] \times \{0\} \times  [0,1]$, which are also called braids. Formally, a braid is described by its braid word in the \emph{braid group} $B_n$, which has generators $\sigma_1, \sigma_2, \ldots, \sigma_{n-1}$ and relations:
\[\sigma_i \sigma_k = \sigma_k \sigma_i \quad |i-k|>1, \qquad \sigma_i \sigma_{i+1} \sigma_i = \sigma_{i+1} \sigma_i \sigma_{i+1} \quad i = 1, \dots, n-2\] 

The relationship between knots and braids is established through the Alexander theorem \cite{A}, which states that every oriented knot or link can be represented isotopically  as the closure of some braid. The {\it (standard) closure} of a geometric braid comprises in joining with simple arcs the corresponding endpoints, and it gives rise to an {\it oriented} knot or link. Figure~\ref{stclosure} illustrates an example of a braid on the left and its closure on the right. The braid on the left is represented by the braid word $\sigma_2 \sigma_1^{-1} \sigma_2^{-1} \sigma_1 \sigma_2^{-1}$, where the generators $\sigma_i$ and their inverses correspond to single crossings of adjacent strands.

\begin{figure}[H]
\begin{center}
\includegraphics[width=2.7in]{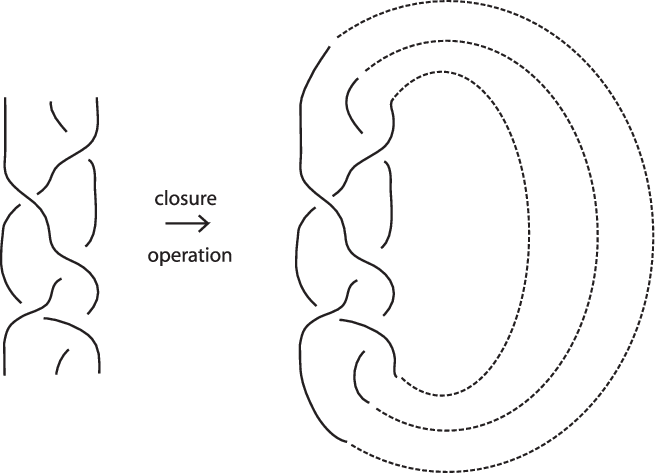}
\end{center}
\caption{The  braid $\sigma_2 \sigma_1^{-1} \sigma_2^{-1} \sigma_1 \sigma_2^{-1}$ and its closure.}
\label{stclosure}
\end{figure}

To determine whether two braids yield isotopic knots or links via their closures, Markov's theorem provides the necessary and sufficient conditions \cite{M}. It states that two braids have equivalent closures if and only if they are related by a finite sequence of:
\begin{itemize}
\item Conjugation: \(\beta \sim \gamma \beta \gamma^{-1}, \quad \beta, \gamma \in B_n\); 
\smallbreak
\item Stabilization move: \(\beta \sim \beta \sigma_n^{\pm 1} \in B_{n+1}, \quad \beta \in B_{n}\).
\end{itemize}
In \cite{LR1} an 1-move analogue of the Markov theorem is formulated using the {\it $L$-moves}, which generalize the stabilization moves and also generate conjugation.

\subsection{Bonded braid definition}\label{sec:bbraids}

In this subsection, we introduce the notion of bonded braids, which are classical braids equipped with bonds. The standard closure operation for bonded braids produces bonded links (see Figure~\ref{closure}). Bonded braids extend the classical braid framework and provide an algebraic structure for studying bonded links. To establish this correspondence, we define bonded braid isotopy, which preserves the bond structure under standard braid moves. Using this framework, we prove the analogue of the Alexander theorem for bonded links. 
\smallbreak
We begin with the formal definition of bonded braids and their isotopy before presenting the braiding algorithm and proof of the Alexander theorem.

\begin{definition}\rm
A {\it bonded braid on $n$ strands} is a pair $(\beta, B)$ of a classical braid $\beta$ on $n$ strands, and a set of $k$ disjoint, embedded horizontal simple arcs, called {\it bonds}. The boundary points of a bond, called {\it nodes}, cannot coincide with endpoints of the braid or with other nodes, and they have local neighborhoods that are 3-valent graphs with the attaching arcs, like $\vdash$, forming together an \texttt{H}-neighborhood, as depicted in Figure~\ref{HN}.
 A bond joining the $i^{th}$ and $j^{th}$ strands with $i<j$ is denoted $b_{i,j}$ and it threads transversely through the strands in between. By abuse of notation, we denote by $b_{i,j}$  a bond joining the $i^{th}$ and $j^{th}$ strands with any sequence of overpasses or underpasses.   A bond $b_{i, i+1}$  joining two consecutive strands $i$ and $i+1$ shall be called {\it elementary bond}, and will be denoted by $b_i$. If $B=\emptyset$, then $(b,\emptyset)$ is just a classical braid. 
 Moreover, a {\it bonded braid diagram} (also referred to as bonded braid) is a regular projection of a bonded braid in the plane $[0,1] \times \{0\} \times  [0,1]$ with over/under information at every crossing, and such that no crossing is horizontally aligned with a bond. So, $b_{i,j}$ can be encoded unambiguously by a sequence of $j-i+1$ {\it o}'s and {\it u}'s indicating the types of crossings formed between the bond and any strand in between. 
 For an example see left-hand side illustration of Figure~\ref{closure}.  A configuration of a bond which passes either over or under all its intermediate strands, such that all crossings are marked `over'/`under', is a {\it uniform over/under}. 
\end{definition}

\begin{figure}[H]
\begin{center}
\includegraphics[width=2.7in]{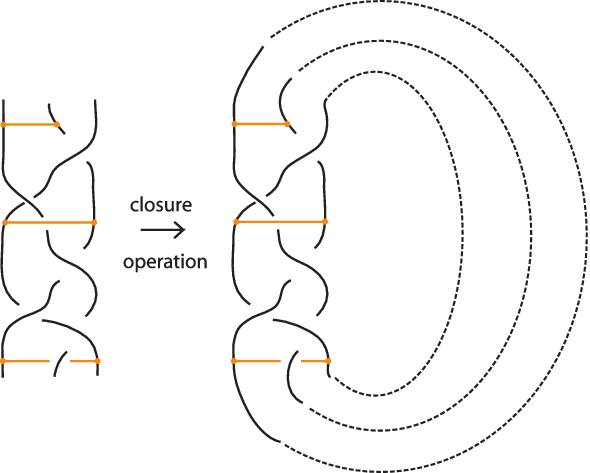}
\end{center}
\caption{A bonded braid and the closure operation.}
\label{closure}
\end{figure}

\subsection{Bonded braid isotopy}

To define bonded braid isotopy, we extend the classical braid isotopy by introducing additional moves that account for the presence of bonds. These moves reflect the topological nature of bonds and their interactions with each other and with crossings and other strands in the braid. The first such move is bonded planar isotopy, induced by the {\it bonded braided $\Delta$-moves} illustrated in Figure~\ref{delta}, where the downward orientation of the strands is preserved. 

\begin{figure}[H]
\begin{center}
\includegraphics[width=1.3in]{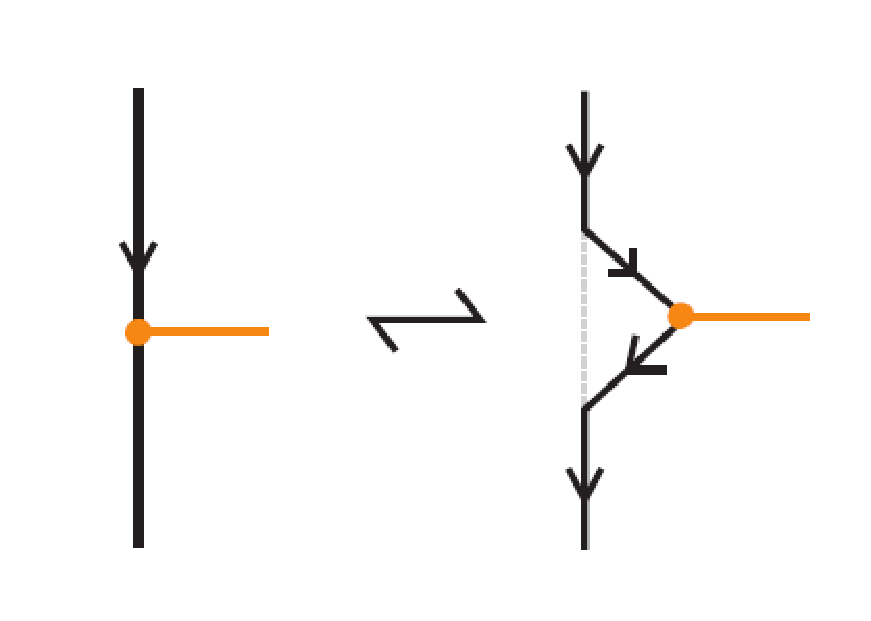}
\end{center}
\caption{A bonded braid planar isotopy move.}
\label{delta}
\end{figure}

The {\it interaction between two bonds} depends on their relative positions and the sequence of crossings they form with the strands in between. More precisely:

\begin{itemize}
\item If two bonds $b_{i,j}$ and $b_{k,l}$ are sufficiently far apart so that their threading paths do not overlap or interact, that is $j<k$, they commute freely. For an example see Figure~\ref{twobonds3}.

\begin{figure}[H]
\begin{center}
\includegraphics[width=3.3in]{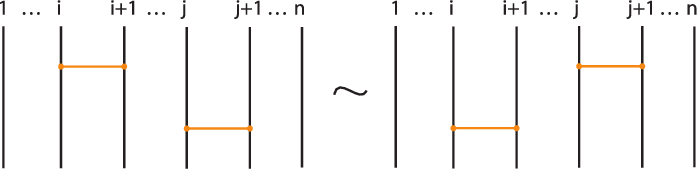}
\end{center}
\caption{Interactions between two distant bonds.}
\label{twobonds3}
\end{figure}

\item Two bonds  $b_{i,j}$ and $b_{k,l}$ also commute if $i<k<l<j$ and $b_{i,j}$ is a uniform over/under, that is, it passes either over or under all its intermediate strands, which include the set of intermediate strands of $b_{k,l}$. In this case, the configuration of the sequence of {\it u}'s and {\it o}'s of the bond $b_{k,l}$ does not matter.  For an example view Figure~\ref{twobonds}.

\begin{figure}[H]
\begin{center}
\includegraphics[width=4.5in]{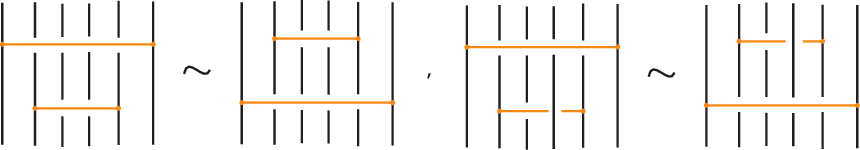}
\end{center}
\caption{Interactions between two bonds: uniform over configurations.}
\label{twobonds}
\end{figure}

\item Furthermore, two bonds $b_{i,j}$ and $b_{k,l}$ commute if $i<k<l<j$,   the sequence of crossings of $b_{k,l}$ coincides with a subsequence  of the sequence of crossings of $b_{i,j}$ for the strands $k,l$ and the markings in the bigger sequence before and after the subsequence agree, i.e. they are both `over' or both `under'.  For example, if $b_{k,l}$ forms the sequence $(o,o,u,o)$ with its intermediate strands, the  bond $b_{i,j}$ must contain either the subsequence $[o,(o,o,u,o),o]$ or the subsequence $[u,(o,o,u,o),u]$ for the bonds to commute. We call this configuration the {\it matching crossing sequences}. For an illustration see Figure~\ref{twobonds2}.

\begin{figure}[H]
\begin{center}
\includegraphics[width=2.2in]{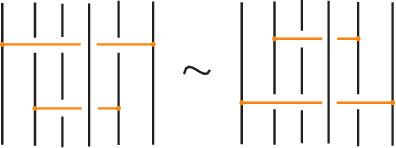}
\end{center}
\caption{Interactions between two bonds: matching crossing sequences.}
\label{twobonds2}
\end{figure}
\end{itemize}

Any {\it interaction between a bond and an arc } is a result of the braided vertex slide moves exemplified in Figure~\ref{regulartotight1}.

\begin{figure}[H]
\begin{center}
\includegraphics[width=4in]{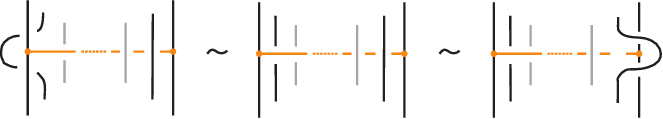}
\end{center}
\caption{Braided vertex slide moves.}
\label{regulartotight1}
\end{figure}

An {\it interaction between a bond  and a crossing} depends on the relative positions of the bond and the strands forming the crossing. More precisely:

\begin{itemize}
\item If the strands $i, i+1$ forming the crossing $\sigma_i$ do not interact with the bond $b_{k,l}$, i.e. $k>i+1$ or $l<i$, then the bond and the crossing commute freely (see Figure~\ref{bondsrel2} for an example).

\begin{figure}[H]
\begin{center}
\includegraphics[width=2.8in]{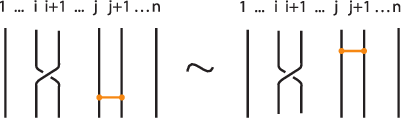}
\end{center}
\caption{A crossing commuting with a bond.}
\label{bondsrel2}
\end{figure}

\item If the  nodes of the bond lie on two strands that form a crossing, the bond and the crossing commute, i.e. $b_i$ commutes with $\sigma_i$  (see Figure~\ref{bondsrel2b}). This is the {\it bonded flype move}. 

\begin{figure}[H]
\begin{center}
\includegraphics[width=2.9in]{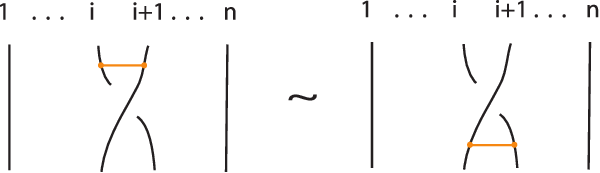}
\end{center}
\caption{A crossing between strands that contain the nodes of a bond.}
\label{bondsrel2b}
\end{figure}

\item A bond $b_{k,l}$ and a crossing $\sigma_i$ commute if the bond passes entirely over or entirely under both strands that form the crossing. We call this configuration, the {\it uniform position of the bond}. For an illustration of this case, see the left-hand side of Figure~\ref{bondcros}. 
\smallbreak
\item A bond and a crossing also commute if the over-strand of the crossing lies  above the bond and the under-strand lies below it. In this case, the bond effectively `threads through' the crossing without disrupting its configuration. We refer to this move as the {\it bonded R3 move}. For an illustration of this case, see the right-hand side of Figure~\ref{bondcros}. 

\begin{figure}[H]
\begin{center}
\includegraphics[width=4.5in]{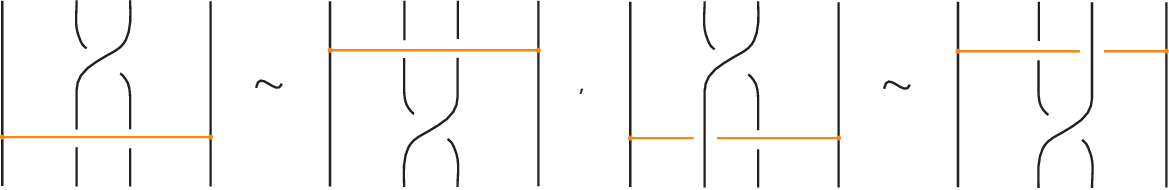}
\end{center}
\caption{A bond passing over a crossing and a bonded R3 move.}
\label{bondcros}
\end{figure}

\item Figures~\ref{bondcros1}, \ref{bondrel2} illustrate the cases where one node of the bond lies on a strand, resp. on two stands, that crosses (resp. cross) another strand. In the instance of Figure \ref{bondrel2} there may be other strands threading through the bond, as long as they all lie under the overcrossing strand.

\begin{figure}[H]
\begin{center}
\includegraphics[width=4.3in]{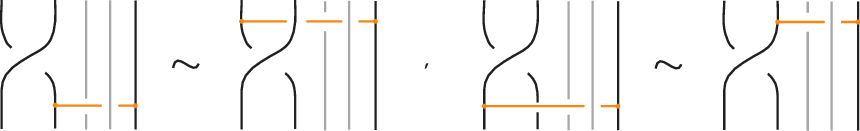}
\end{center}
\caption{Interactions between a crossing and a bond with one node on a strand of the crossing.}
\label{bondcros1}
\end{figure}

\begin{figure}[H]
\begin{center}
\includegraphics[width=3in]{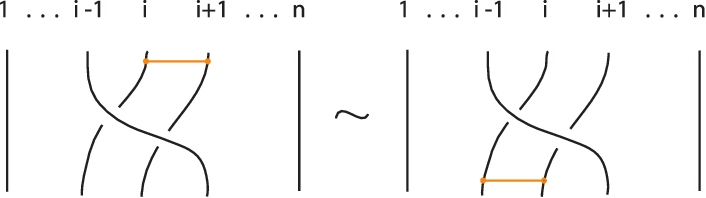}
\end{center}
\caption{A braid strand passing over a bond.}
\label{bondrel2}
\end{figure}
\end{itemize}

\begin{lemma}\label{lem:basicbbmove}
A move between a crossing and a bond with one node on a strand of the crossing follows from a move where a braid strand passes over or under a bond, together with the other moves, and vice versa.
\end{lemma}
\begin{proof} The proof of the Lemma is adequately described in Figure~\ref{slidemove}. \end{proof}

\begin{figure}[H]
\begin{center}
\includegraphics[width=3.5in]{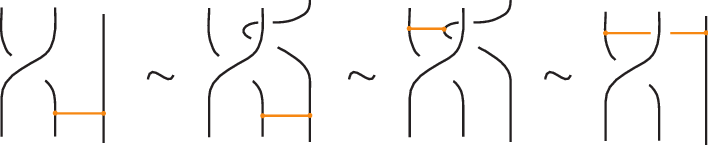}
\end{center}
\caption{The  proof of  Lemma~\ref{lem:basicbbmove}.}
\label{slidemove}
\end{figure}

Finally, in Figure~\ref{forbond}, we illustrate some examples of configurations that are not allowed under bonded braid equivalence.

\begin{figure}[H]
\begin{center}
\includegraphics[width=4.3in]{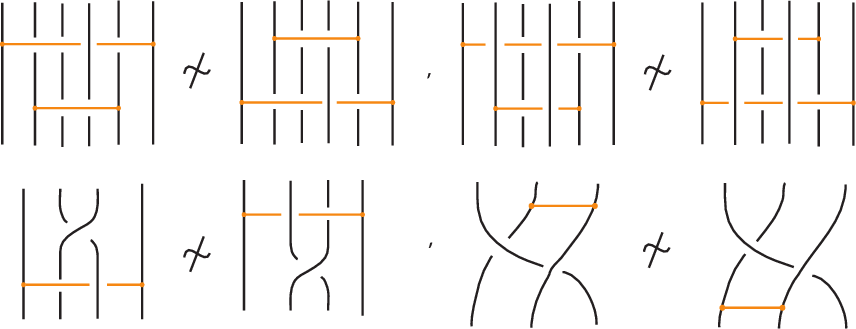}
\end{center}
\caption{Examples of moves that are not allowed in bonded braid equivalence. These configurations violate the topological constraints of the bonded braid model.}
\label{forbond}
\end{figure}

The additional moves depicted in Figures~\ref{twobonds}, \ref{twobonds2} and \ref{bondcros} are essential to define bonded braid equivalence. These moves account for the interactions specific to bonds, ensuring that the topological properties of the bonded braid are preserved. Using these moves, we extend classical braid isotopy to include bonded braid isotopy, leading to the following formal definition.

\begin{definition}\rm
Two bonded braids are {\it isotopic} if and only if they differ by classical braid isotopy that takes place away from the bonds, together with moves depicted in Figures~\ref{twobonds}--\ref{bondcros}. An equivalence class of isotopic bonded braid diagrams is called a {\it bonded braid}.
\end{definition}

\subsection{The Bonded Braid Monoid}

Let now $BB_n$ be the set of bonded braids on $n$-strands. In $BB_n$ we define as operation the usual concatenation of classical braids. Then, the set of bonded braids forms a monoid, called the {\it bonded braid monoid}. We denote by $b_{i, j}$ any bond whose nodes lie on the $i^{th}$ and $j^{th}$ strand of the braid, and by $b_i$ we denote the bond $b_{i,i+1}$  whose nodes lie on the $i^{th}$ and $(i+1)^{th}$ strands, which we shall call {\it elementary bond}, in contrast to the {\it long bonds} $b_{i, j}$ for $i \neq j$. We are now ready to give a presentation of \( BB_n \).

\begin{theorem}\label{thm:bondedmonoid}
The set \( BB_n \) of bonded braids on \( n \) strands forms a monoid under braid concatenation. The bonded braid monoid has a presentation with generators:
\[
\sigma_1, \ldots, \sigma_{n-1} \quad (\text{classical invertible braid generators}) 
\quad \text{and} \quad b_{i,j} \quad (1\leq i<j\leq n) \quad (\text{bonds}),
\]
subject to the relations:
\begin{itemize}
\item[-] the classical braid relations among the \( \sigma_i \)'s,
\item[-] the relations illustrated in Figures~\ref{twobonds3}–\ref{bondrel2} (including all their variants), describing all possible interactions of bonds among themselves and with arcs or crossings.
\end{itemize}
\end{theorem}
\begin{proof}
The proof of the Theorem is an immediate consequence of the exhaustive listing above of all possible interactions of bonds among themselves as well as with other arcs or crossings in the bonded braid. 
\end{proof}

We further  have the following: 

\begin{lemma}\label{lembw}
A standard bond is a word of the classical braid generators and an elementary bond.
\end{lemma}

\begin{proof}
Figure~\ref{gb2} illustrates a bonded braid isotopy for contracting the bond or equivalently for pulling to the side strands crossing over or under the bond, using braided vertex slide moves  (recall Figure~\ref{regulartotight1}). \end{proof}

\begin{figure}[H]
\begin{center}
\includegraphics[width=3.8in]{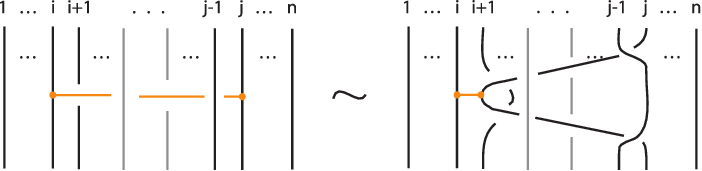}
\end{center}
\caption{The bond $b_{i,j}$ expressed as a combination of classical braid generators and an elementary bond.}
\label{gb2}
\end{figure}

Using now Lemma~\ref{lembw} and standard Tietze transformations on the presentation of \( BB_n \) of Theorem~\ref{thm:bondedmonoid}, one can obtain a reduced presentation for \( BB_n \) that uses only the classical braid generators \( \sigma_i \) and the elementary bonds \( b_i \).  Namely, we have the following:

\begin{theorem}\label{thm:tightbonded}
The bonded braid monoid, $BB_n$, on \( n \) strands admits the following reduced presentation: it is generated by the classical braid generators and their inverses  \( \sigma_1^{\pm1}, \ldots, \sigma_{n-1}^{\pm1} \) and the elementary bonds \( b_1, \ldots, b_{n-1} \), subject to the relations:
\[
\begin{array}{rcll}
\sigma_i\, \sigma_j & = & \sigma_j\, \sigma_i & \quad \text{for } \ |i-j|>1, \\[4pt]
\sigma_i\, \sigma_{i+1}\, \sigma_i  & = &  \sigma_{i+1}\, \sigma_i\, \sigma_{i+1} & \quad \text{for all } \ i, \\[4pt]
b_i\, b_j  & = &  b_j\, b_i & \quad \text{for } |i-j|>1, \\[4pt]
b_i\, {\sigma_j}  & = &  {\sigma_j} \, b_i & \quad \text{for } \ |i-j|>1, \\[4pt]
b_i\, \ {\sigma_i}  & = &  {\sigma_i} \, b_i & \quad \text{for all } \ i, \\[4pt]
b_i\, \sigma_{i+1}\, \sigma_i  & = &  \sigma_{i+1}\, \sigma_i\, b_{i+1} & \quad \text{for all } \ i,\\[4pt]
 \sigma_i \, \sigma_{i+1} \, b_i & = &  b_{i+1} \, \sigma_i \, \sigma_{i+1}  & \quad \text{for all } \ i.
\end{array}
\]
\end{theorem}

\noindent When considering $BB_n$ with its reduced presentation we shall refer to it as the \emph{tight bonded braid monoid}. 

We now observe that the algebraic structure of the bonded braid monoid is closely related to the well-studied singular braid monoid \cite{Baez} and \cite{Bi}. In particular, one can interpret each bond between two strands as a kind of singular crossing, where the two arcs are  tangential instead of forming a crossing, as illustrated in Figure~\ref{sing1}.

\begin{figure}[H]
\begin{center}
\includegraphics[width=1.3in]{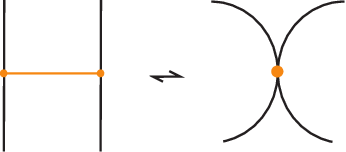}
\end{center}
\caption{Bonds as singular crossings.}
\label{sing1}
\end{figure}

For reference, the singular braid monoid \( SB_n \) is generated by the standard braid generators and their inverses   \( \sigma_1^{\pm1}, \ldots, \sigma_{n-1}^{\pm1} \) together with singular crossing generators \( \tau_1, \ldots, \tau_{n-1} \), subject to the following relations:
\[
\begin{array}{rcll}
\sigma_i\, \sigma_j  & = &  \sigma_j\, \sigma_i & \text{for } |i-j|>1, \\[4pt]
\sigma_i\, \sigma_{i+1}\, \sigma_i  & = &  \sigma_{i+1}\, \sigma_i\, \sigma_{i+1} & \text{for all } i, \\[4pt]
\tau_i\, \tau_j  & = &  \tau_j\, \tau_i & \text{for } |i-j|>1, \\[4pt]
\tau_i\, {\sigma_j}  & = &  {\sigma_j} \, \tau_i & \text{for } |i-j|>1, \\[4pt]
\tau_i\, {\sigma_i} & = &  {\sigma_i} \, \tau_i & \text{for all } i, \\
[4pt]
\tau_i\, \sigma_{i+1}\, \sigma_i  & = &  \sigma_{i+1}\, \sigma_i\, \tau_{i+1} & \text{for all } i \\[4pt]
 \sigma_i \, \sigma_{i+1} \, \tau_i & = &  \tau_{i+1} \, \sigma_i \, \sigma_{i+1}  & \text{for all } \ i. \\[4pt]
\end{array}
\]

Comparing the above presentation of \( SB_n \)  with the reduced presentation of \( BB_n \) in Theorem~\ref{thm:tightbonded} leads to the following result:

\begin{theorem}\label{singbondedmonoid}
The bonded braid monoid \( BB_n \) is isomorphic to the singular braid monoid \( SB_n \).
\end{theorem}

\begin{proof}
By identifying each bond generator \( b_i \) with the singular crossing generator \( \tau_i \) as illustrated in Figure~\ref{sing1}, we see that the defining relations of \( BB_n \) coincide with those of \( SB_n \). Hence, the assignment \( b_i \mapsto \tau_i \) and \( \sigma_i \mapsto \sigma_i \) extends to a monoid isomorphism \( BB_n \cong SB_n \).
\end{proof}

The  above leads also to the following remark:

\begin{remark}
A singular knot or braid can be realized geometrically with tangential singularities in place of singular crossings, giving rise to a new theory.  Regarding the singularities as tangentialities of lines, one can analyze the dynamics of the birth and death of such interactions in a generalization of the present work.
\end{remark}

\begin{remark}
We observe  that the first two relations in Theorem~\ref{thm:tightbonded} are the relations of  the classical braid group $B_n$.  So we have that the  classical braid group $B_n$ injects in the tight bonded braid monoid $BB_n$. This follows  from Theorem~\ref{singbondedmonoid} and from the analogous  result about  the singular braid monoid \( SB_n \).    
 In general, by virtue of Theorem~\ref{singbondedmonoid} any result on the singular braid monoid \( SB_n \)  transfers intact to the tight bonded braid monoid $BB_n$. And vice versa, like Theorem~\ref{thm:irredundant} that follows, which is also valid for \( SB_n \). We further observe  that the relations satisfied by the $b_i$'s are compatible with the braid relations. Therefore, there is a surjection from  $BB_n$ to $B_n$ by assigning \( b_i \mapsto \sigma_i \) and \( \sigma_i \mapsto \sigma_i \) and another one by assigning \( b_i \mapsto id \) and \( \sigma_i \mapsto \sigma_i \).  
\end{remark}

We can further reduce the presentation of \( BB_n \) by using only the classical braid generators \( \sigma_i \) and a single bond generator \( b_1 \). Indeed, by any of the two last relations in Theorem~\ref{thm:tightbonded} all elementary bonds \( b_i \) for \( i>1 \) can be expressed as conjugates of \( b_1 \) by the appropriate braid words. Suppose we fix  the first ones: $b_i\, \sigma_{i+1}\, \sigma_i  =  \sigma_{i+1}\, \sigma_i\, b_{i+1}$. Then we obtain: 
\[
\begin{array}{rcl}
b_2 & = & (\sigma_2 \sigma_1)^{-1}  \, b_1\, (\sigma_2 \sigma_1) \\[4pt]
b_3  & = & (\sigma_3 \sigma_2)^{-1}  (\sigma_2 \sigma_1)^{-1}  \, b_1\, (\sigma_2 \sigma_1) (\sigma_3 \sigma_2)   \\[4pt]
& \vdots &  \\[4pt]
 b_i & = & (\sigma_i \sigma_{i-1})^{-1} \cdots (\sigma_2 \sigma_1)^{-1}  \, b_1\, (\sigma_2 \sigma_1) \cdots  (\sigma_i \sigma_{i-1}),
  \\[4pt]
& \vdots & 
\end{array}
\]
On the other hand, by the last relations in Theorem~\ref{thm:tightbonded} we obtain: $b_2 =  (\sigma_1\, \sigma_2) \, b_1 \,  (\sigma_1\, \sigma_2)^{-1}$. So, substituting $b_2$ in the first of the above relations, we extract the relation: 
\[ 
 b_1\, (\sigma_2 \sigma_1)(\sigma_1\, \sigma_2) = (\sigma_2 \sigma_1)(\sigma_1\, \sigma_2) \, b_1
\]
Therefore, by the above and applying  Tietze transformations on the relations of  Theorem~\ref{thm:tightbonded} we obtain the following irredundant presentation:

\begin{theorem}\label{thm:irredundant}
The tight bonded braid monoid \( BB_n \) admits the following irredundant presentation: it is generated by the classical braid generators  and their inverses   \( \sigma_1^{\pm1}, \ldots, \sigma_{n-1}^{\pm1} \) and a single bond generator \( b_1 \), subject to the relations:
\[
\begin{array}{rcll}
\sigma_i\, \sigma_j & = & \sigma_j\, \sigma_i & \quad \text{for } |i-j|>1, \\[4pt]
\sigma_i\, \sigma_{i+1}\, \sigma_i  & = &  \sigma_{i+1}\, \sigma_i\, \sigma_{i+1} & \quad \text{for all } i, \\[4pt]
b_1\, {\sigma_j} & = & {\sigma_j}  \, b_1 & \quad \text{ for } j=1 \text{ \& } j>2, \\[4pt]
 b_1\, (\sigma_2 \sigma_1)(\sigma_1\, \sigma_2) & = &  (\sigma_2 \sigma_1)(\sigma_1\, \sigma_2) \, b_1 & \\[4pt] 
b_1 \, (\sigma_2 \sigma_1 \sigma_3 \sigma_2)^{-1}  \, b_1\, (\sigma_2 \sigma_1 \sigma_3 \sigma_2) & = &  (\sigma_2 \sigma_1 \sigma_3 \sigma_2)^{-1}  \, b_1\, (\sigma_2 \sigma_1 \sigma_3 \sigma_2) \, b_1  &
\end{array}
\]
\end{theorem}

\section{An analogue of the Alexander Theorem for bonded links}\label{alsec}

 As noted in Section~\ref{sec6:bondedbraids}, one can apply to a bonded braid the usual  closure operation and obtain a bonded link (recall Figure~\ref{closure}). In this section we focus on the theory of topological standard bonded links and we present a braiding algorithm for oriented standard bonded links.  Namely, we prove the following bonded analogue of the Alexander Theorem for classical links:

\begin{theorem}[{\bf Braiding theorem for bonded links}] \label{alexbl}
 Every oriented topological standard bonded link can be represented topological vertex isotopically as the closure of a standard resp. tight bonded braid. 
\end{theorem}

\begin{proof}
 For proving the theorem we adapt to the bonded setting the  braiding algorithm used in \cite{LR1} for braiding classical oriented link diagrams. The main idea was to keep the downward arcs (with respect to the top-to-bottom direction of the plane) of the oriented link diagrams fixed and to replace upward arcs with braid strands. Upward arcs may need to be further subdivided  into smaller arcs, each passing either over or under other arcs,  as shown in Figure~\ref{upa}, which we label with an `o' or a `u' accordingly.   These final upward  arcs are called {\it up-arcs}.  If an up-arc contains no crossings, then the choice is arbitrary. 

\begin{figure}[H]
\begin{center}
\includegraphics[width=3.7in]{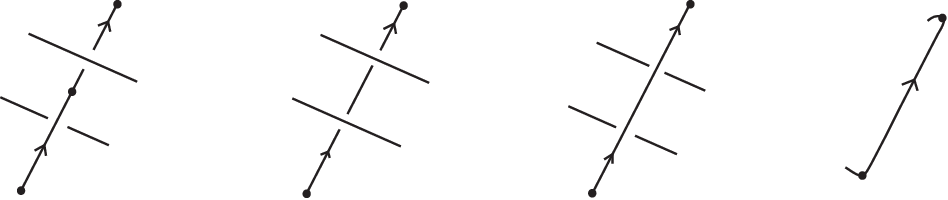}
\end{center}
\caption{Upward oriented arcs are further subdivided to up-arcs, entirely ‘over’ or ‘under’.}
\label{upa}
\end{figure}

 We  assume, by general position, that there are no horizontal or vertical link arcs and that no two subdividing points or local maxima and minima or crossings are vertically aligned and no crossing is horizontally aligned with a bond. We then  proceed with applying the {\it braiding moves}, as illustrated abstractly in Figure~\ref{ahg}. We perform an $o$-braiding move on an up-arc  with an `o' label, whereby the new pair of corresponding braid strands replacing the up-arc run both entirely over the rest of the diagram  (see Figure~\ref{ahg}), and analogous $u$-braiding moves on the up-arcs  with a `u' label. 

\begin{figure}[H]
    \centering
    \includegraphics[width = .6\textwidth]{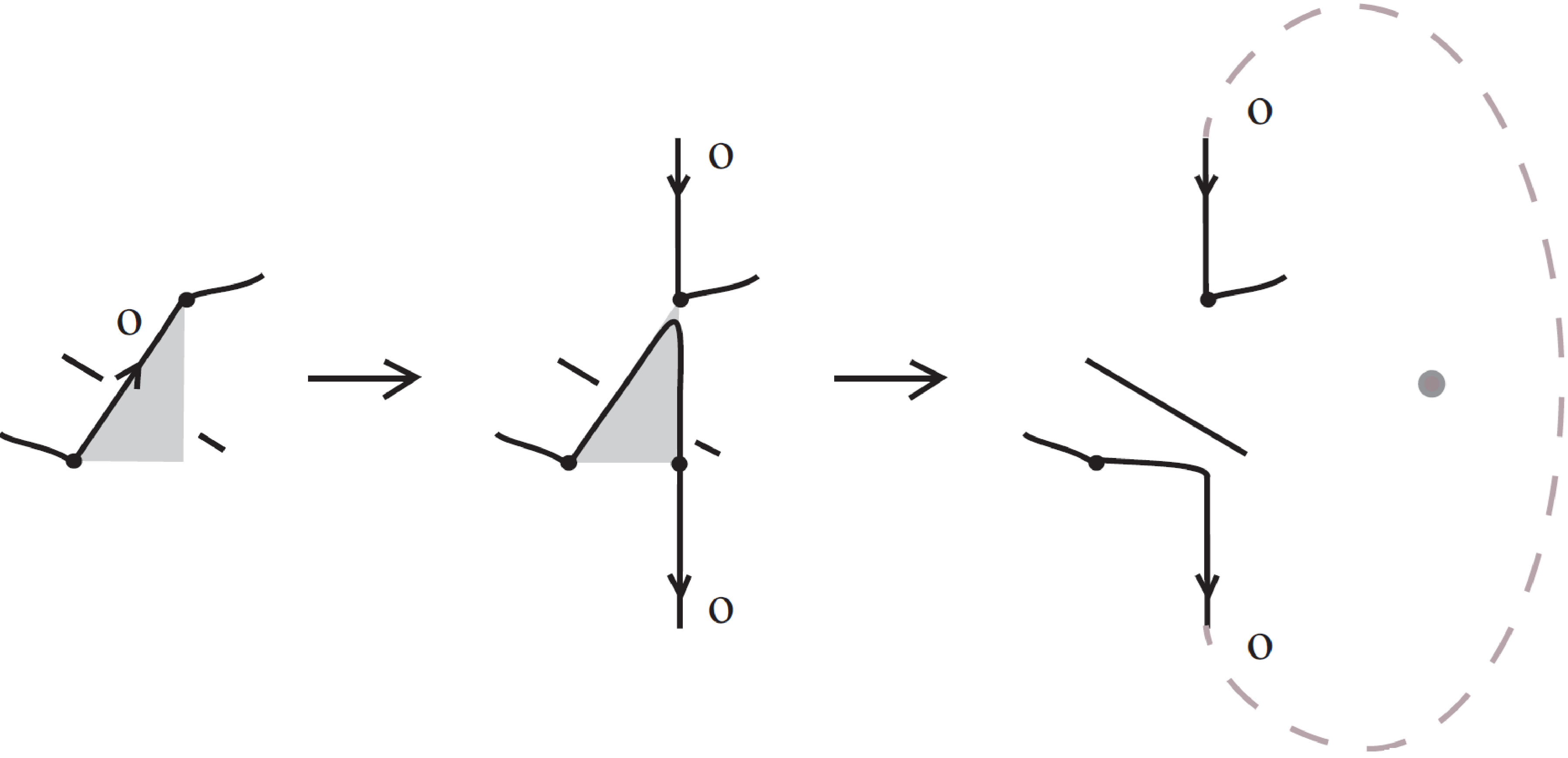}
    \caption{Braiding moves for up-arcs.}
\label{ahg}
\end{figure}

In order to use this classical braiding algorithm on standard/tight bonded links we need to deal further with situations containing bonds.  In the {\it bonded general position}, we also require that a node is not vertically aligned with subdividing points or local maxima and minima. 

\smallbreak
\noindent {\it Step 1:} We first bring all bonds to horizontal position using (bonded) planar isotopy, observing the general position rules. View Figure~\ref{verticalbond}.  

 \begin{figure}[H]
\begin{center}
\includegraphics[width=3.1in]{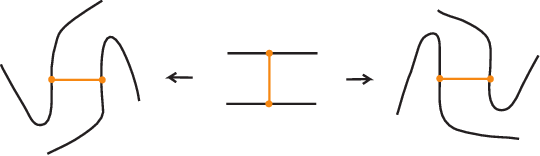}
\end{center}
\caption{Bringing bonds to horizontal position.}
\label{verticalbond}
\end{figure}

\noindent {\it Step 2:} We next deal with up-arcs in the diagram that contain at least one node of a bond, standard or tight. For this we apply the isotopy moves TVT or RVT as exemplified in Figures~\ref{orbond} and \ref{orbond1}, according to whether the two joining arcs are antiparallel or parallel upward arcs. Note that  the braiding preparation of a bonded region with two parallel upward arcs can be also achieved with planar isotopy, by  performing a 180-degree rotation on the plane. Then the bonds will only lie on down-arcs, so that the braiding algorithm will not affect them.  

\begin{figure}[H]
\begin{center}
\includegraphics[width=3.2in]{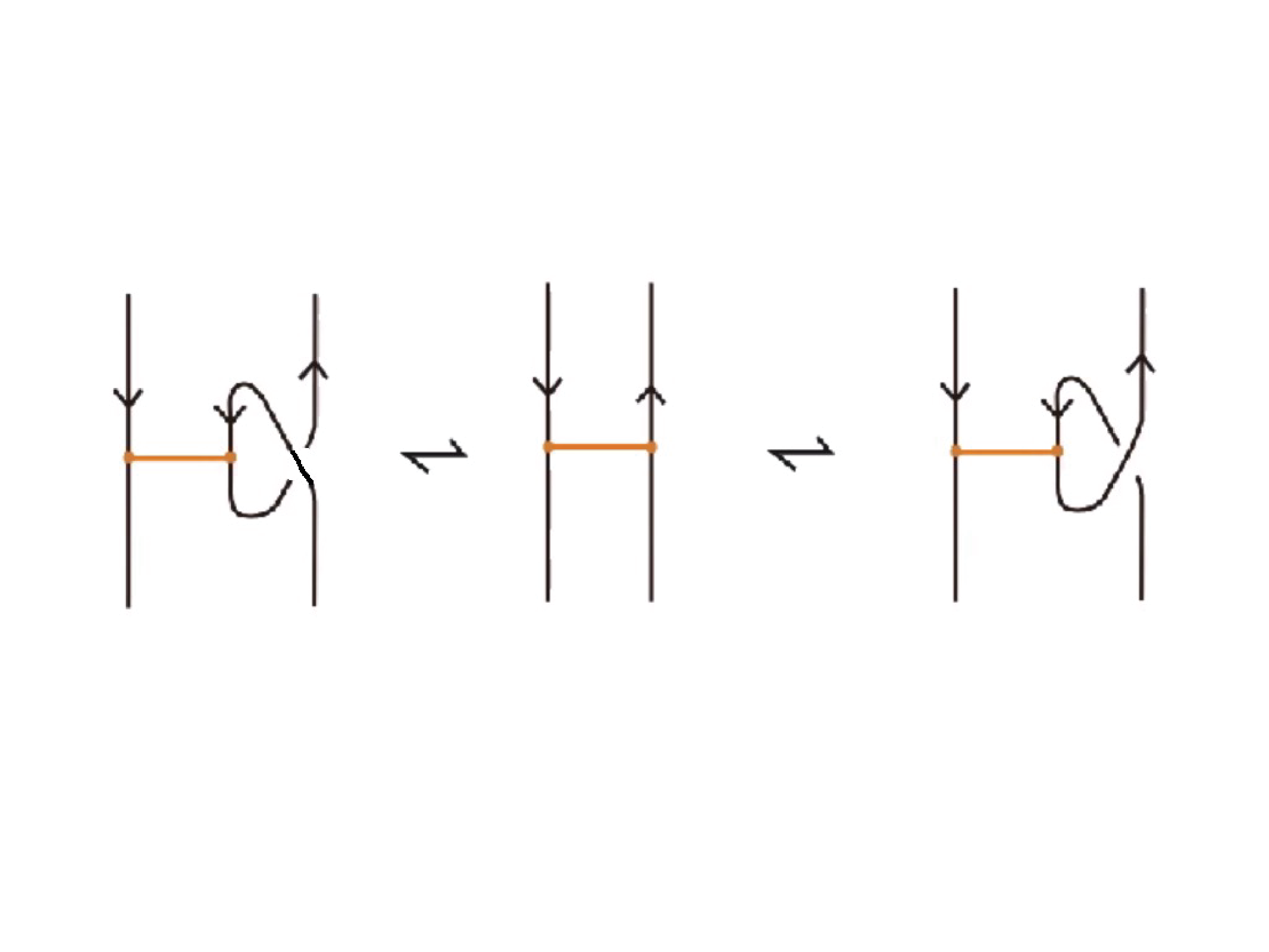}
\end{center}
\caption{ Braiding preparation for two antiparallel bonded arcs using TVT-moves.}
\label{orbond}
\end{figure}

\begin{figure}[H]
\begin{center}
\includegraphics[width=2in]{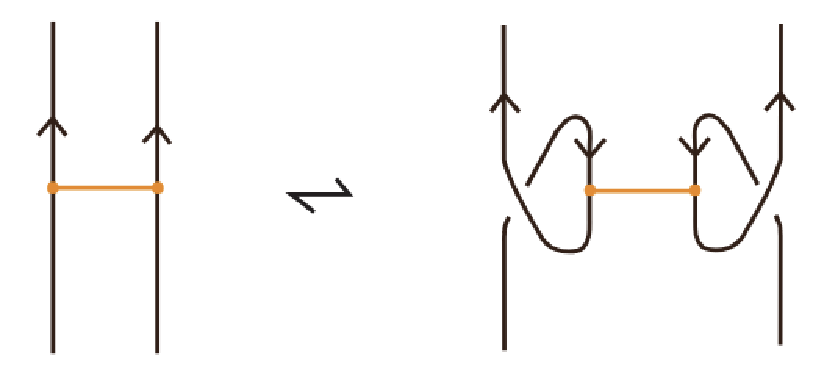}
\end{center}
\caption{Braiding preparation for two parallel bonded arcs oriented upward using RVT-moves.}
\label{orbond1}
\end{figure}

\noindent {\it Step 3:}  Finally, we apply the braiding algorithm of \cite{LR1} for the link diagram, ignoring the bonds. After eliminating all up-arcs we obtain a standard bonded braid, the closure of which is -by construction- isotopic to the original oriented bonded link diagram. 
\smallbreak

\noindent {\it Step 4:}  In the end of the process some braid strands may cross over or under some bonds. Using the braided vertex slide moves (recall Figure~\ref{gb2}) we bring them to tight form. 
The algorithm above provides a proof of Theorem~\ref{alexbl}.
\end{proof}

\begin{remarks}\rm\,
\begin{itemize}
\item[(a)] Note that in the rigid vertex category it is not possible to braid two anti-parallel strands that are connected by a bond. This is because braiding such strands necessarily requires the use of a topological vertex twist at one of the nodes, a move that is forbidden in the rigid vertex category. For this reason we restricted our focus to the topological vertex category. 

\item[(b)] The vertex slide moves for bringing the bonds to tight form (recall Figure~\ref{regulartotight}) could have been already applied  after Step 1 or Step 2. In the end we might still need to apply the braided vertex slide moves for new braid strands.
\end{itemize}
\end{remarks}


\section{Bonded braid equivalences}\label{mtbb}

In this section we formulate and prove two bonded braid equivalences for topological bonded braids: the bonded $L$-equivalence for standard bonded braids and  the tight bonded $L$-equivalence for tight bonded braids.  For proving a braid equivalence we need to have the diagrams in  general position, as described in the proof of the bonded braiding algorithm, to which we now add the {\it triangle condition}, whereby any two sliding triangles of up-arcs do not intersect. This is achieved by further subdivision of up-arcs, if needed, and it allows any order in the sequence of the braiding moves. We note here that the interior of a sliding triangle may intersect a bond.

We proceed with introducing the notion of $L$-moves in the bonded setting.In the classical setting, the $L$-moves naturally generalize the stabilization moves in the Markov Theorem for classical braids, since an $L$–move is equivalent to adding in a braid a positive or a negative crossing, so that two braids that differ by an $L$-move have isotopic closures. Moreover, as shown in \cite{LR1}, the $L$-moves can also realize conjugation for classical braids.  On the other hand, an $L$–move can be created by braid isotopy, stabilization and conjugation \cite{LR1}. The $L$–move approach to Markov-type theorems is flexible and powerful for formulating braid equivalences in practically any topological setting. They prove to be particularly useful in settings where the  braid analogues do not have an apparent algebraic structure. In \cite{LR1, La, KaLa} $L$-moves and braid equivalence theorems are presented for different knot theories.  

\begin{definition}\label{lmdefn}\rm
An {\it $L$-move} on a bonded braid $\beta$, consists in the following operation or its inverse: we cut an arc of $\beta$ open and we pull the upper cutpoint downward and the lower upward, so as to create a new pair of braid strands with corresponding endpoints (on the vertical line of the cutpoint), and such that both strands run entirely {\it over} or entirely {\it under}  the rest of the braid (including the bonds). Pulling the new strands over will give rise to an {\it $L_o$-move\/} and pulling under to an {\it  $L_u$-move\/}. By definition, an $L$-move cannot occur in a bond or at a node. By a general position argument, the new pair of strands does not pass by a crossing or a node. Figure~\ref{Lb} shows $L_o$ and $L_u$ moves taking place above or below a bond, on its  attaching strand. 

Furthermore, a {\it tight $L$-move} on a tight bonded braid is defined as above, with the addition that, if the new strands  are vertically aligned with a bond, the strands can be pulled to the side, right or left,  using the braided vertex slide moves as described in Figure~\ref{gb2}, see also Figure~\ref{Lbcross}, so as to remain in the tight category. A pulled away $L$-move in the tight category shall be called {\it tight $L$-move}. 
 Tight $L$-moves include $L$-moves that do not cross bonds. 

$L$-moves (resp. tight $L$-moves) together with bonded braid isotopy generate an equivalence relation in the set of all bonded braids (resp. tight bonded braids), the {\it $L$-equivalence} (resp. tight $L$-equivalence).
\end{definition}

\begin{figure}[H]
\begin{center}
\includegraphics[width=4.9in]{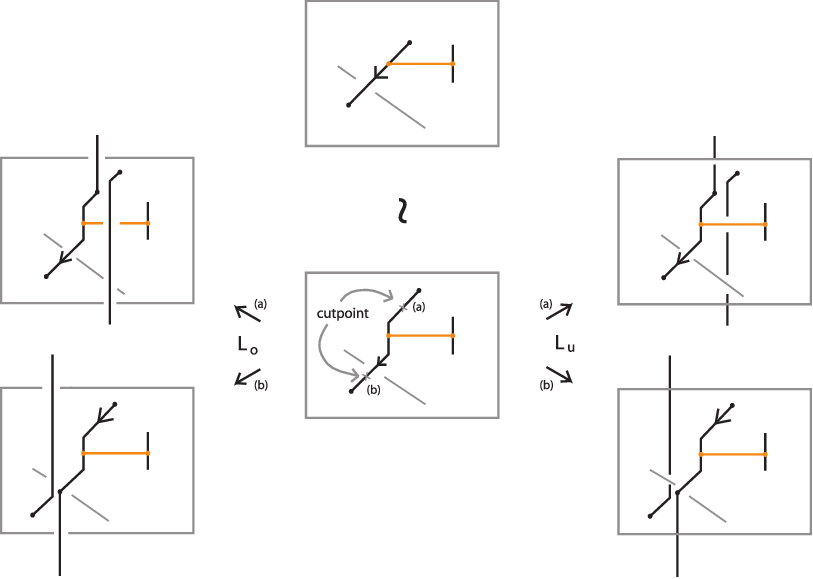}
\end{center}
\caption{$L$-moves on a strand with a bond.}
\label{Lb}
\end{figure}

\begin{note}\label{contractingL}
We note that the closures of two bonded braids that differ by  an $L$-move are rigid vertex isotopic. See second illustration of Figure~\ref{Lbcross}, where the closure of the $L$-move contracts to the original arc. 
\end{note}

\begin{note}\label{inboxcrossing}
Like in the classical setting, an $L$-move is equivalent to introducing a crossing in the bonded braid formed by the new pair of strands, which can be pulled to the (far) right (or left) of the bonded braid using bonded braid isotopy. View Figure~\ref{Lbcross}. Note that in the closure, this formation contracts to an R1 move. 
\end{note}

\begin{figure}[H]
\begin{center}
\includegraphics[width=5.2in]{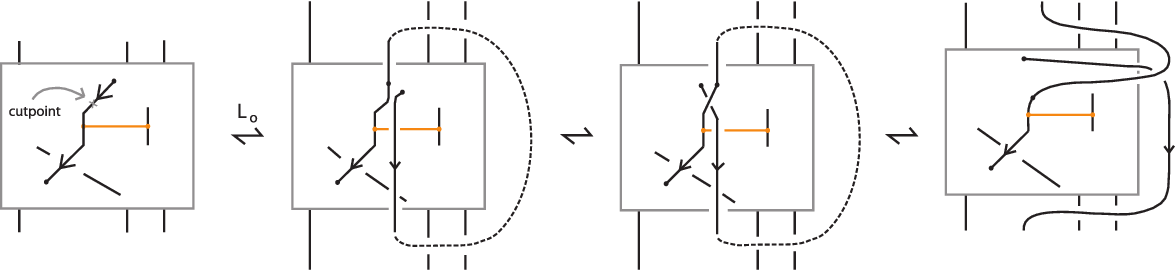}
\end{center}
\caption{A crossing formed after an $L$-move is performed, and pulling the strands of an $L$-move  away from the bond.}
\label{Lbcross}
\end{figure}

We consider now the left hand  and the right hand side bonded diagrams  in Figure~\ref{orbond}, which differ only by the one crossing which can be positive or negative.  We braid the two bonded diagrams using the braiding moves described in the proof of Theorem~\ref{alexbl}. Examining carefully the resulting bonded braids leads to the following definition.

\begin{definition}\label{bondedL}\rm
A {\it bonded $L$-move} between two bonded braids resembles an $L$-move with an in-box crossing. More precisely, a {\it bonded $L_o$-move} (resp. a {\it bonded $L_u$-move})  consists in the following operation or its inverse: we cut an arc adjacent to a bond node and create with the two cutpoints a crossing of specific type. We then pull the two ends, the upper upward  and the lower downward, so as to create a new pair of vertically aligned braid strands, such that both strands run entirely {\it over} (resp. entirely {\it under}) the rest of the bonded braid (including the bonds). The choice of the  crossing is determined by the following property: a vertex slide move with the arc of the crossing not adjacent to the bond node cannot give rise to a classical $L_o$-move with a crossing (resp. a classical $L_u$-move with a crossing). More precisely, if the two cutpoints lie in the upper right arc of the H-region of a bond, the crossing is positive, while if they lie in the lower right arc of the H-region, the crossing is negative. See  Figure~\ref{bondedL}. If the two cutpoints lie in the upper left arc of the H-region of a bond, the crossing is negative, while if they lie in the lower left arc, the crossing is positive. 

Furthermore, a {\it tight bonded  $L$-move} is a bonded  $L$-move defined in analogy to a tight  $L$-move.
\end{definition} 

\begin{figure}[H]
\begin{center}
\includegraphics[width=5.2in]{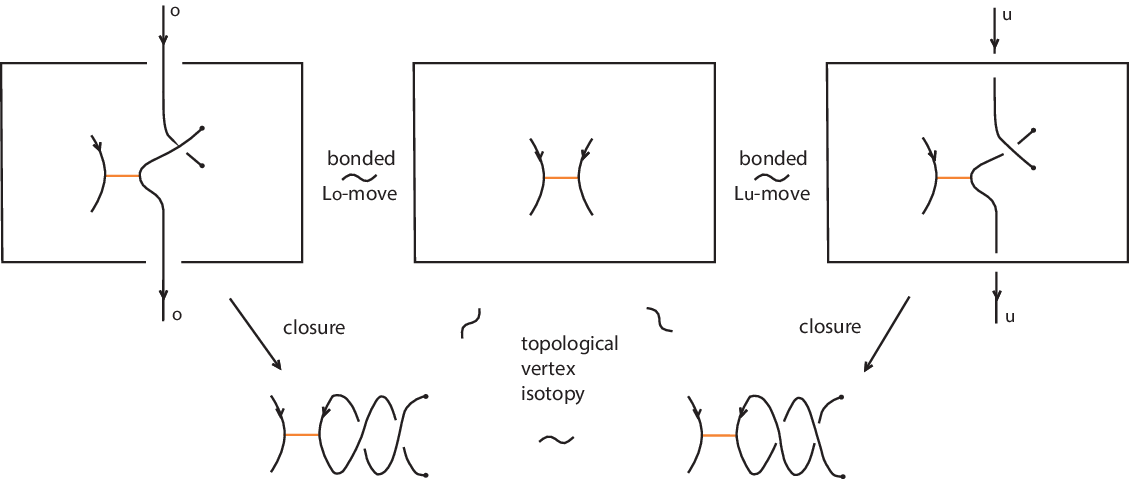}
\end{center}
\caption{Bonded $L$-moves.}
\label{bondedL1}
\end{figure}

\begin{note}\label{bondedL_closure}
As indcated in Figure~\ref{bondedL1}, the closures of two bonded braids that differ by a bonded $L$-move  are topologically vertex equivalent.  
\end{note}

\begin{definition}\label{bondedLequiv}\rm
$L$-moves (resp. tight $L$-moves)  together with bonded $L$-moves (resp. tight bonded $L$-moves) and bonded braid isotopy generate an equivalence relation in the set of all bonded braids (resp. tight bonded braids), the {\it bonded $L$-equivalence} (resp. {\it tight bonded $L$-equivalence}).
\end{definition} 

We are now in a position to state one of the main results of our paper.

\begin{theorem}[{\bf Bonded $L$-equivalence for topological bonded braids}] \label{markbll}
 Two bonded braids upon closure give rise  to topologically vertex isotopic oriented standard bonded links  if and only if they can be obtained one from the other by a finite sequence of bonded braid isotopy  and the following moves:
\[
\begin{array}{lllcll}
1. & L-moves &  &  &  & \\
2.  & Bonded \ L-moves &  &  &  & \\
3. & Bond \ Commuting: &  \alpha\, b_{i,j} & \sim & b_{i,j} \, \alpha, & {for}\ \alpha,\, b_{i,j} \in BB_n\\
\end{array}
\]
Furthermore, two tight bonded braids upon closure give rise  to topologically vertex isotopic oriented tight bonded links  if and only if they can be obtained one from the other by a finite sequence of tight bonded braid isotopy  and the following moves:
\[
\begin{array}{lllcll}
1. & {Tight \ L-moves} &  &  &  &\\
2.  & Tight \ bonded \ L-{moves} &  &  &  & \\
3. & {Elementary \ Bond \ Commuting:} &  \alpha\, b_i & \sim & b_i \, \alpha, & {for}\ \alpha,\, b_i \in BB_n\\ 
\end{array}
\]
\end{theorem}

The bond commuting is illustrated in Figure~\ref{bcm}.
\begin{figure}[H]
\begin{center}
\includegraphics[width=3in]{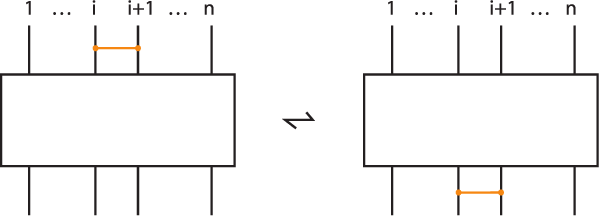}
\end{center}
\caption{Tight bond commuting.}
\label{bcm}
\end{figure}

\begin{proof} 
For the one direction, the closures of two standard/tight bonded braids that differ by the moves of either statement are clearly topologically vertex equivalent. Indeed, 
the closures of $L$-moves are discussed in  Notes~\ref{contractingL} and~\ref{inboxcrossing} (recall Figure~\ref{Lbcross}), for the closures of bonded $L$-moves recall Figure~\ref{bondedL1} and Note~\ref{bondedL_closure}, while bonded commuting is realized via planar isotopy.

For the converse, in order to ensure that the stated moves are sufficient we need to examine any choices made  for bringing a bonded diagram to general position and during the braiding algorithm, and show that they result in bonded $L$-equivalent bonded braids. Similarly, that  any bonded isotopy moves on a bonded diagram correspond to bonded $L$-equivalent bonded braids. We shall only examine choices involving bonds. All others are proved as in the classical case \cite{LR1}.

The first choice made for bringing a bonded diagram to general position is when bringing a vertical bond to the horizontal position, recall Figure~\ref{verticalbond}. Let $D_1, D_2$ be two oriented bonded diagrams that differ by one such move, from the one horizontal position to the other.  Figure~\ref{verticalbond_equiv} demonstrates the $L$-equivalence of the corresponding bonded braids for the case of parallel attaching arcs at the nodes, after braiding the regions of the two nodes. The other cases are proved likewise. The arrow indication on the bond is placed for facilitating the reader in following the different directions. Note that, if some arcs cross the bonds, these can be pulled away in both diagrams using the same vertex slide moves, so they will be braided identically. Therefore, the moves can be assumed to be local, and that all other up-arcs in both diagrams are braided, so that we can compare in the figure the final braids.

\begin{figure}[H]
\begin{center}
\includegraphics[width=5.5in]{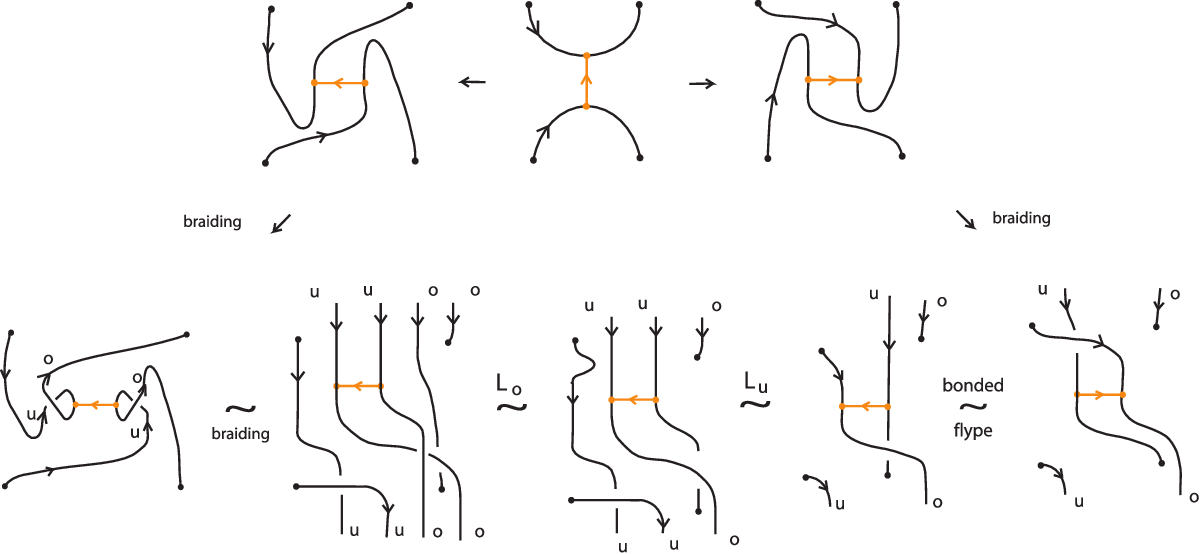}
\end{center}
\caption{The choice of bringing a vertical bond to horizontal gives rise to $L$-equivalent bonded braids.}
\label{verticalbond_equiv}
\end{figure}

Another choice made during the braiding algorithm is when applying a TVT-move, as exemplified in Figure~\ref{orbond}, so as to prepare our diagram for braiding. The crossing involved in the move can be positive or negative. One can then easily verify that the resulting braids are bonded $L$-equivalent.

Let now $D_1, D_2$ be two bonded diagrams that differ by a topological bonded isotopy move.  For classical planar isotopies and the classical Reidemeister moves the reader is referred to \cite{LR1}, as the proofs pass intact in this setting. Note that the bond commuting move as well as the move we examined above, with tight bonds, comprise  bonded planar isotopy moves. The moves of type Reidemeister 2 and~3 of Theorem~\ref{reid_standard} with one bond, all end up being invisible in the bonded braid after completing the braiding.  Suppose next that $D_1, D_2$ differ by a topological vertex twist (TVT) move. We have two cases (if we overlook the signs of the crossings), illustrated in Figure~\ref{orientedTVT}.

\begin{figure}[H]
\begin{center}
\includegraphics[width=4.2in]{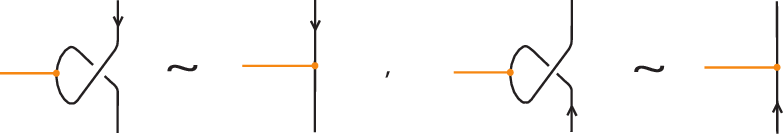}
\end{center}
\caption{The two types of oriented TVT moves.}
\label{orientedTVT}
\end{figure}

The first type of oriented TVT (left hand side of Figure~\ref{orientedTVT}) is treated in  Figure~\ref{orientedTVTa}, where we first braid the region of the node in the one diagram. By the triangle condition we may assume that the rest of the diagrams are braided, so that in the figure we compare the final bonded braids. 

\begin{figure}[H]
\begin{center}
\includegraphics[width=5.5in]{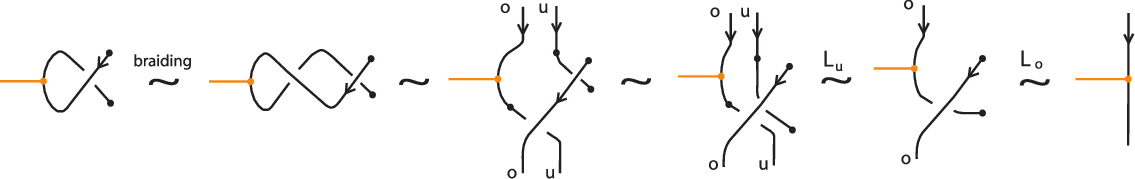}
\end{center}
\caption{$L$-equivalence of one type of an oriented TVT move.}
\label{orientedTVTa}
\end{figure}

\noindent The second type of oriented TVT (right hand side of Figure~\ref{orientedTVT}) is straightforward as shown in  Figure~\ref{orientedTVTb}, where we braid the region of the node in the one diagram using the same crossing as in the other diagram.

\begin{figure}[H]
\begin{center}
\includegraphics[width=2in]{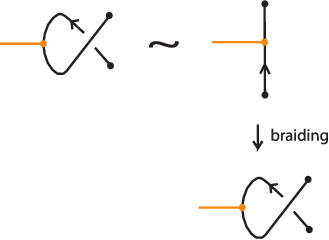}
\end{center}
\caption{Braiding consistently the node in the oriented TVT move results in identical diagrams.}
\label{orientedTVTb}
\end{figure}

We finally check the vertex slide (VS) moves. The three representing cases of oriented VS moves with the middle arc being an up-arc are illustrated in Figures~\ref{VDdown}, \ref{VSantiparallel} and \ref{VSup}. The moves are considered local so that all other braiding is done and we can compare the final braids. Figure~\ref{VDdown} shows the case of parallel down-arcs at the bonding sites. After the performance of the $L$-moves we see that the move is invisible on the bonded braid level. 

\begin{figure}[H]
\begin{center}
\includegraphics[width=4.1in]{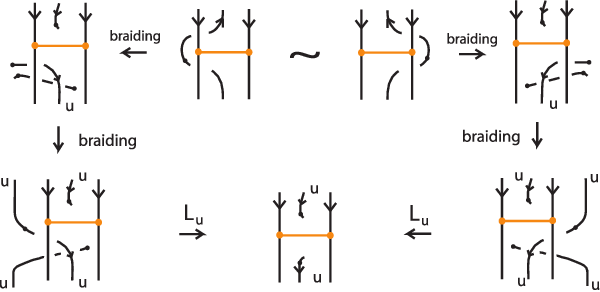}
\end{center}
\caption{An oriented VS move with parallel down-arcs and its braiding analysis.}
\label{VDdown}
\end{figure}

\noindent Figure~\ref{VSantiparallel} shows the case of antiparallel arcs at the bonding sites. After braiding the up-arc with a TVT move, we perform a classical R3 move, which ensures $L$-equivalent bonded braids, and we arrive at the formation of the previous case.

\begin{figure}[H]
\begin{center}
\includegraphics[width=3.8in]{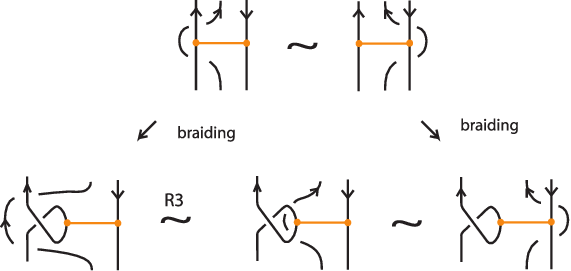}
\end{center}
\caption{An oriented VS move with antiparallel down-arcs and its braiding analysis.}
\label{VSantiparallel}
\end{figure}

\noindent Finally, Figure~\ref{VSup} shows the case of parallel up-arcs at the bonding sites. Clearly this move rests also on the previous cases after braiding  up-arcs at the nodes with  TVT moves and performing  R3 moves.

\begin{figure}[H]
\begin{center}
\includegraphics[width=1.7in]{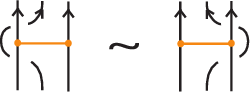}
\end{center}
\caption{The braiding analysis of an oriented VS move with parallel up-arcs reduces to the previous cases.}
\label{VSup}
\end{figure}

\noindent Note that, if some arcs cross the bonds in either one of the three cases, these can be pulled away in both diagrams involved in the move, using the same vertex slide moves, so they will be braided identically. In the course of proving this Theorem it becomes apparent that it suffices to assume that all bonds are contracted to tight bonds. Likewise  all $L$-moves can be assumed to be tight $L$-moves by bonded braid isotopy.  From the proof above it follows that, eventually, we only need to check the moves of Theorem~\ref{tightequiv}. 

We have checked all moves of Theorems~\ref{reid_standard} and~\ref{tightequiv}, so the proof of both statements is completed.
\end{proof}

\begin{note}
It is important to emphasize on the fact that {\it bond commuting} (in the closure of bonded braids) constitute equivalence moves for bonded braids, since these moves are not captured by the $L$-moves. The situation is in direct comparison with the $L$-move equivalence for singular braids, where the commuting of a singular crossing is imposed in the equivalence (cf. \cite{La}).
\end{note}

The $L$-moves naturally generalize the stabilization moves for classical braids, since an $L$–move is equivalent to adding a positive or a negative crossing (recall Note~\ref{inboxcrossing} and Figure~\ref{Lbcross}). Moreover, as shown in \cite{LR1}, the $L$-moves can also realize conjugation for classical braids and this carries through also to bonded braids. On the other hand, as demonstrated in Figure~\ref{Lbcross}, an $L$–move can be created by braid isotopy, stabilization and conjugation. Hence, we may replace moves (1) of Theorem~\ref{markbll} and, along with the algebraization of the bonded $L$-moves, we can obtain the analogue of the Markov theorem for bonded braids. This will be discussed in a further paper.

\section{The theory of enhanced bonded links \& braids}\label{fbb}

We now extend the bonded knot theory by introducing \emph{enhanced bonds}. In an \emph{enhanced bonded link}, each bond is assigned one of two types, conventionally termed \emph{attracting} vs.\ \emph{repelling} bonds (see Figure~\ref{fblink1}). These types are meant to model different physical interactions: for example, an attracting bond might represent a bond arising from an attractive force (like a disulfide bond pulling two parts of a protein together), whereas a repelling bond may model an effective constraint that prevents two regions from approaching too closely, for example due to steric hindrance or electrostatic repulsion. 
 Topologically, we can regard an attracting or a repelling bond as the same kind of embedded arc as before,  but  we treat them as  `inverses' of one another in an algebraic sense. We will formalize this by allowing bonds to be cancelling inverses on the level of enhanced bonded braids. 

In this section, we first define enhanced bonded links and the moves that relate them. Then we introduce the enhanced bonded braid group, denoted $EB_n$, which extends the bonded braid monoid by adding inverses for the bond generators. We then describe analogues of the Alexander and Markov theorems and $L$-equivalence in the enhanced context. We believe that the theory of enhanced bonded links can serve as a better model for biological knotted objects, as well as in other physical situations.

\subsection{Enhanced Bonds: Attracting vs.\ Repelling}

\begin{definition} \label{forcedbld}\rm 
An {\it (oriented) enhanced bonded knot/link}  is a pair $(L, B)$, where $L$ is an (oriented) link in $S^3$, and $B = B_+ \cup B_-$ is a finite set of bonds as defined in Definition~\ref{def:bonded} and such that each element of the subset $B_+$ is equipped  with two arrows pointing toward one another, and each element of the subset $B_-$ is equipped  with two arrows pointing away from one another. Call elements of $B_+$  \emph{attracting  bonds} and elements of $B_-$  \emph{repelling bonds}. For an example view Figure~\ref{fblink1} where the enhanced bonds are standard. Moreover, two adjacent bonds are called {\it parallel} if together they form two opposite sides of a standardly embedded rectangular band, so that the other two sides are link arcs, and such the band can be contracted by isotopy to a regularly projected planar tight rectangle (in the sense that its interior does not intersect other arcs or bonded arcs). Further, if we have two adjacent parallel bonds of opposite type, then they cancel, meaning that the pair of bonds can be removed. 
An {\it enhanced bonded link diagram}  is a diagram of an (oriented) enhanced bonded knot/link.
\end{definition}  

\begin{figure}[H]
\begin{center}
\includegraphics[width=4.4in]{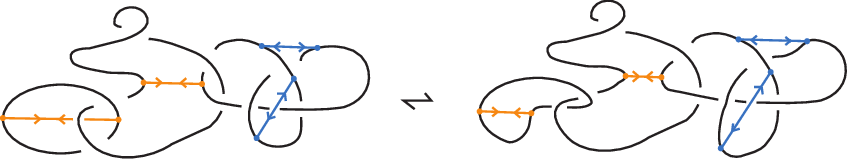}
\end{center}
\caption{Interpretation of enhanced bonds as `attracting' and `repelling'.}
\label{fblink1}
\end{figure}

Regarding  isotopy, the moves we introduced for bonded links do not alter the types assigned to the bonds in enhanced bonded links. So, {\it enhanced bonded isotopy} is defined in the same way as bonded isotopy, with the additional requirements that the bond type (attracting or repelling) assigned to each bond remains fixed throughout the isotopy, and that  parallel bonds of opposite type cancel, view Figure~\ref{bondcancllation}. Both isotopy categories, topological and rigid vertex, continue to apply in this setting, as do the three diagrammatic forms: long, standard, and tight enhanced bonded links. 

\begin{figure}[H]
\begin{center}
\includegraphics[width=1.5in]{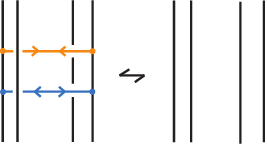}
\end{center}
\caption{Cancellation of parallel bonds of opposite type.}
\label{bondcancllation}
\end{figure}

\begin{remark}
Note that it is possible to generalize the definition of enhanced bonded link diagrams to include bonds that are without the above form of enhancement (so that they do not have inverses) and also one can include many types of bonds, with or without inverses, if these are needed. We retain the above definition in our discussion for the sake of simplicity.
\end{remark}

\subsection{The enhanced bonded braid group}\label{ebbgr}

Recall that bonded braids are defined as classical braids equipped with embedded horizontal simple arcs, the bonds. Similarly, {\it enhanced bonded braids} are classical braids equipped with two different types of bonds, the attracting  and the  repelling bonds. By abuse of notation, we shall denote by $b_i$ an attracting {\it elementary enhanced bond}  and by $b_i^{-1}$ we shall denote the corresponding repelling  bond. Moreover, we will denote by $b_{i,j}$ an attracting bond  between the $i^{th}$ and $j^{th}$ strands of an enhanced bonded braid and by $b_{i,j}^{-1}$ the corresponding repelling  bond, such that when we have  two consecutive bonds of different types between the same strands of an enhanced bonded braid, they cancel out (see Figures~\ref{bondcancllation} and \ref{fb2}). 

\smallbreak
Recall  that the singular braid monoid $SB_n$ is isomorphic to the bonded braid monoid, $BB_n$ (Theorem~\ref{singbondedmonoid}). Also that, as shown in \cite{FKR}, \( SB_n \) embeds into a group, called the singular braid group. This proves that  the argument in \cite{FKR} extends to the bonded braid context, so that the monoid also embeds into a group, corresponding the  two distinct types of enhanced bonds, attracting and repelling, to the  two distinct types of marked singular crossings, as illustrated in Figure~\ref{sing}.

\begin{figure}[H]
\begin{center}
\includegraphics[width=3.8in]{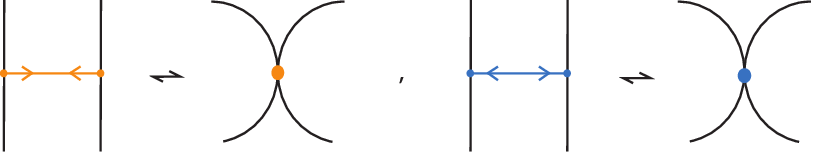}
\end{center}
\caption{The bonds as marked singular crossings.}
\label{sing}
\end{figure}

Hence, we have shown that:
\begin{theorem}

The bonded braid monoid embeds into a group,  the \emph{enhanced bonded braid group}.
\end{theorem}

The above results lead to the definition of the {\it enhanced bonded braid group} $EB_n$ with  operation the  concatenation of bonded braids in the set of enhanced bonded braids on $n$-strand.  Note that as in the case of bonded braids, we will denote by $b_{i, j}$ all possible sequences of $j-i+1$ {\it o}'s and {\it u}'s, that indicate the type of the crossings formed between the enhanced bond and the strands of the braid, and by $b_{i, j}^{-1}$ its corresponding inverse in the enhanced bonded braid group. Finally, isotopy between two enhanced bonded braids is defined in the same way as isotopy between bonded braids, since these isotopy moves do not alter the type on the enhanced bonds. Then with the same  reasoning as earlier we obtain:

\begin{theorem}\label{fbbmon}\rm
The {\it  enhanced tight bonded braid group} $EB_{n}$ is the group generated by the classical invertible braid generators \( \sigma_1, \ldots, \sigma_{n-1} \), the {\it attracting bonds}  $b_1, \ldots, b_{n-1}$ and the {\it repelling bonds}  $b_1^{- 1}, \ldots, b_{n-1}^{- 1}$, collectively called {\it enhanced bonds} (see Figure~\ref{fb1}), with operation the usual braid concatenation. The generators satisfy the same relations as the generators of the tight bonded braid monoid, $BB_{n}$, together with the following relations:
\[
b_i\, b_i^{-1}\ =\ 1\ =\ \ b_i^{-1}\, b_i,\  {\rm for\ all}\ i.
\]
\end{theorem}

\begin{figure}[H]
\begin{center}
\includegraphics[width=2.9in]{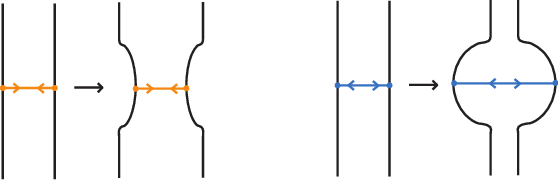}
\end{center}
\caption{The enhanced bonds $b_i$ and $b_i^{-1}$.}
\label{fb1}
\end{figure}

In Figure~\ref{fb2} we demonstrate the relations $b_i\, b_i^{-1}\ =\ 1\ =\ b_i^{-1}\, b_i$.

\begin{figure}[H]
\begin{center}
\includegraphics[width=1.9in]{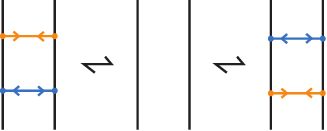}
\end{center}
\caption{The relations $b_i\, b_i^{-1}\ =\ 1\ =\ b_i^{-1}\, b_i$.}
\label{fb2}
\end{figure}

Note that when two enhanced bonds cancel the underlying knot is not affected by the cancellation.

We further note that, as in the tight bonded braid monoid  (recall Theorem~\ref{thm:irredundant}), the enhanced tight bonded braid group \( EB_n \) also admits an irredundant presentation:

\begin{theorem}\label{thm:irredundant_enh}
The enhanced tight bonded braid group \( EB_n \) admits an irredundant presentation with generators the classical invertible braid generators \( \sigma_1, \ldots, \sigma_{n-1} \) together with a single attracting bond generator \( b_1 \) and its inverse \( b_1^{-1} \), subject to the relations:
\[
\begin{array}{rcll}
\sigma_i\, \sigma_j & = & \sigma_j\, \sigma_i & \quad \text{for } |i-j|>1, \\[4pt]
\sigma_i\, \sigma_{i+1}\, \sigma_i  & = &  \sigma_{i+1}\, \sigma_i\, \sigma_{i+1} & \quad \text{for all } i, \\[4pt]
b_1\, {\sigma_j} & = & {\sigma_j}  \, b_1 & \quad \text{ for } j=1 \text{ \& } j>2, \\[4pt]
 b_1\, (\sigma_2 \sigma_1)(\sigma_1\, \sigma_2) & = &  (\sigma_2 \sigma_1)(\sigma_1\, \sigma_2) \, b_1 & \\[4pt] 
b_1 \, (\sigma_2 \sigma_1 \sigma_3 \sigma_2)^{-1}  \, b_1\, (\sigma_2 \sigma_1 \sigma_3 \sigma_2) & = &  (\sigma_2 \sigma_1 \sigma_3 \sigma_2)^{-1}  \, b_1\, (\sigma_2 \sigma_1 \sigma_3 \sigma_2) \, b_1  &
\end{array}
\]
\end{theorem}

\subsection{Analogues of Alexander and Markov theorems for enhanced bonded braids}\label{almarebb}

In this subsection we present the analogues of the Alexander and the Markov theorems, as well as the $L$-equivalence for enhanced bonded braids. The main idea is that the bonded isotopy moves remain the same as for usual bonds, so enhanced bonds respect the braiding algorithm presented in \S~\ref{alsec}, and  the results we obtained for the bonded $L$-equivalence in \S~\ref{mtbb} also carry through. Hence we have the following:

\begin{theorem}[{\bf Braiding theorem for enhanced bonded links}] \label{falexbl}
 Every oriented topological standard enhanced bonded link can be represented isotopically as the closure of a standard resp. tight enhanced bonded braid. 
\end{theorem}

For Markov's theorem, now that bonds have inverses, one might wonder if we can allow a stabilization that introduces a bond together with its inverse. For instance, in an enhanced context, one could imagine a stabilization that adds both a $b_n^+$ and a $b_n^-$ such that their effect cancels in closure. Indeed, one might define a ``cancellation move'' where an attracting bond and a repelling bond that connect the same pair of points on a link can be removed. In the algebraic setting, that is exactly the relation $b_i^+ b_i^- = 1$, which means one may remove a $b^+$ and $b^-$ pair if they occur adjacent in the word. 

Before moving toward a braid equivalence theorem for enhanced bonded braids, we first need to generalize the notion of the $L$-moves on this setting. We define an $L$-move on an enhanced bonded braid as in Definitions~\ref{lmdefn} and \ref{bondedL}, by ignoring the types of the bonds (recall Figure~\ref{Lb}). It follows from the results \S~\ref{mtbb} that with the use of the $L$-moves on enhanced bonded braids, we obtain the $L$-move equivalence for enhanced bonded braids, by treating enhanced bonds as  simple bonds. Namely, we obtain the following:

\begin{theorem}[{\bf Bonded $L$-equivalence for enhanced bonded braids}] \label{fmarkbll}
Two enhanced bonded braids upon closure give rise  to topologically vertex isotopic oriented enhanced standard bonded links  if and only if they can be obtained one from the other by a finite sequence of enhanced bonded braid isotopy  and the following moves: 
\[
\begin{array}{lllcll}
1. & L-{moves} &  &  &  & \\
2.  & {Bonded} \ L-{moves} &  &  &  & \\
3. & {Enhanced\ Bond\ Conjugation:} &  \alpha & \sim & b_{i,j}^{\pm 1}\, \alpha\, b_{i,j}^{\mp 1},  & {for}\ \alpha,\, b_{i,j} \in EB_n \\ 
\end{array}
\]
Furthermore, two enhanced tight bonded braids upon closure give rise  to topologically vertex isotopic oriented enhanced tight bonded links  if and only if they can be obtained one from the other by a finite sequence of tight bonded braid isotopy  and the following moves: 
\[
\begin{array}{lllcll}
1. & {Tight \ L-moves} &  &  &  &\\
2.  & Tight \ Bonded \ L-{moves} &  &  &  & \\ 
3. & {Enhanced\ Elementary \ Bond\ Conjugation:} &  \alpha & \sim & b_i^{\pm 1}\, \alpha\, b_i^{\mp 1},  & {for}\ \alpha,\, b_i \in EB_n \\ 
\end{array}
\]
\end{theorem}

We can replace moves (1) of Theorem~\ref{fmarkbll} by considering stabilization moves and conjugation for classical braids and, along with the algebraization of the bonded $L$-moves, we obtain the analogue of the Markov theorem for enhanced bonded braids. This is the subject of a sequel work.

\section{The theory of (enhanced) bonded knotoids and braidoids} \label{sectl}

In this section we review the notion of a bonded knotoid and we extend it to the case of having long bonds. We further introduce the notion of bonded braidoids and their extensions to  enhanced bonded knotoids and braidoids. We start by recalling some results from \cite{T} and \cite{GGLDSK} on knotoids and bonded knotoids respectively.

\subsection{Bonded knotoids}

Knotoids were introduced by Turaev in \cite{T} as a generalization of 1-1 tangles by allowing the endpoints to be in different regions of the diagram. Namely: a {\it knotoid diagram} $K$ in an oriented surface $\Sigma$ is a generic immersion of the unit interval $[0, 1]$ into  
$\Sigma$ whose only singularities are transversal double points endowed with over/undercrossing data called crossings. The images of $0$ and $1$ under this immersion are called the endpoints of $K$ (leg and head of $K$ respectively) and are distinct from each other and from the double points (see Figure~\ref{mkoid} by ignoring the bonds). 

A {\it knotoid } in $\Sigma$ is then an equivalence class of knotoid diagrams in $\Sigma$ up to the equivalence relation induced by the Reidemeister moves (recall Figures~\ref{breid1}--\ref{breid2}), that take place away from the endpoints of the knotoid. A  {\it multi-knotoid diagram} in  $\Sigma$ is a generic immersion of the unit interval $[0, 1]$ and a number of copies of $S^1$ into  
$\Sigma$.  Note that in the theory of multi-knotoids, the pulling of a strand that is adjacent to an endpoint is not allowed, since this would result into trivial knotoid diagrams. These moves are illustrated in Figure~\ref{forb}, and they are called {\it forbidden moves} of multi-knotoids.
 A knotoid has a natural orientation from leg to head. A multi-knotoid diagram is {\it oriented } if orientations are also assigned to its closed components. 

\begin{figure}[H]
\begin{center}
\includegraphics[width=3in]{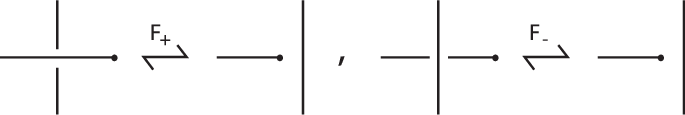}
\end{center}
\caption{The forbidden moves.}
\label{forb}
\end{figure}

{\it Bonded knotoids} were introduced in \cite{GGLDSK} as classical knotoids equipped with standard bonds. Here we we extend the definition to the case of multi-knotoids  having also long bonds. More precisely: 

\begin{definition}\rm
A {\it bonded multi-knotoid diagram} in $\Sigma$ is a pair $(K,B)$ of a  knotoid $K$ equipped with a set $B$ of bonds, defined as in the case of bonded links in Definition~\ref{def:bonded}. View Figure~\ref{mkoid}. Further, a {\it bonded linkoid diagram} is defined to be an immersion of a disjoint union of finitely many unit intervals whose images are knotoid diagrams, equipped with bonds.
\end{definition}  

\begin{figure}[H]
\begin{center}
\includegraphics[width=2.8in]{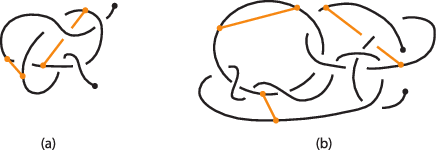}
\end{center}
\caption{A bonded knotoid (a) and a bonded multi-knotoid (b).}
\label{mkoid}
\end{figure}

As in the case of bonded links, bonded multi-knotoid diagrams in $\Sigma$ are liable to two types of equivalence relations: the {\it topological vertex equivalence} and the {\it rigid vertex equivalence}, induced by the moves of  topological vertex isotopy resp. rigid vertex isotopy (recall Proposition~\ref{breidthm}),  all taking place away from the endpoints of the bonded multi-knotoid. 

\begin{definition}\rm
An  equivalence class of bonded multi-knotoid diagrams in $\Sigma$ up to topological vertex equivalence is called a {\it topological bonded multi-knotoid} and up to rigid vertex equivalence it is called a {\it rigid bonded multi-knotoid}. 
\end{definition} 

In the theory of bonded knotoids w have the forbidden moves for the multi-knotoid arcs as well as for the bonds (see Figure~\ref{forb1}). We call these moves {\it bonded forbidden moves}.

\begin{figure}[H]
\begin{center}
\includegraphics[width=3.7in]{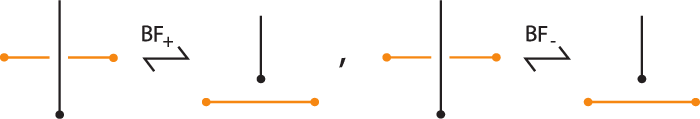}
\end{center}
\caption{The bonded forbidden moves.}
\label{forb1}
\end{figure}

\subsection{Types of closures for bonded multi-knotoids}

As noted in \cite{T}, the theory of knotoids suggests a new diagrammatic approach to the theory of knots, since a knotoid diagram gives rise to a classical knot if we connect its endpoints  with an extra simple arc. Similarly, for bonded multi-knotoids we define the following closure operations: 

\begin{definition}\rm
We call the bonded knot obtained by connecting the endpoints of a bonded knotoid diagram with an arc that goes under or over each arc and each bond it meets, the {\it underpass closure} and the {\it overpass closure} of the bonded knotoid respectively (see Figure~\ref{cloid1}).
\end{definition}

\begin{figure}[H]
\begin{center}
\includegraphics[width=4in]{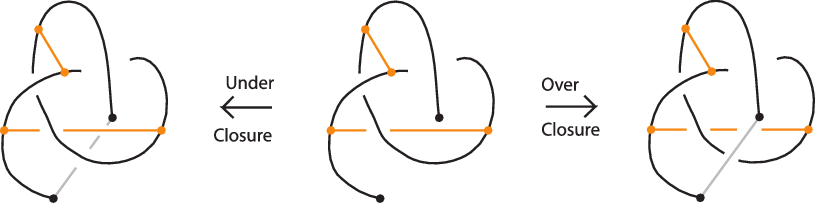}
\end{center}
\caption{The underpass and the overpass closures of a bonded knotoid.}
\label{cloid1}
\end{figure}

Obviously, different closures of a bonded multi-knotoid may result in non isotopic bonded links.  On the other hand, we can represent bonded links via bonded multi-knotoid diagrams, by cutting any arc and declaring the two ends as endpoints. If we fix the closure type we have the following result (proved as in \cite{T} for the case of knotoids): 

\begin{proposition}
Assuming a specific closure type, there is a well-defined surjective map from planar (or spherical) bonded multi-knotoids to bonded links.
\end{proposition}

\begin{remark}\rm
In \cite{MLK} the authors introduce the notion of singular knotoids, that is, knotoid diagrams that contain also singular crossings, and they present a different type of closure operation for singular knotoids, the {\it singular closure}, whereby the closing arcs crosses singularly every arc it intersects. In analogy to the definition of singular closure (under certain conditions) in \cite{D1} the {\it pseudo closure} of knotoids is defined, where a pseudo knotoid diagram is a knotoid with some missing crossing information.
\end{remark}

Likewise, in the theory of bonded multi-knotoids, a different ``closure'' operation may be considered, whereby we introduce a bond whose nodes coincide with the endpoints of a bonded multi-knotoid. More precisely, we have the following definition:

\begin{definition}\rm 
The operation whereby the endpoints of a bonded multi-knotoid diagram are  connected with a standard bond shall be called {\it bonded closure}. In particular, if the connecting bond passes under resp. over each arc and each bond it meets, then we have the {\it underpass  bonded closure}, resp. the {\it overpass bonded closure} of the bonded multi-knotoid  (see Figure~\ref{fakecl}).
\end{definition}

\begin{figure}[H]
\begin{center}
\includegraphics[width=4.2in]{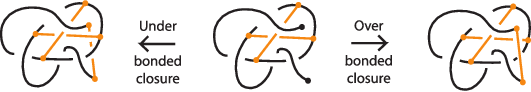}
\end{center}
\caption{Bonded closures of a bonded knotoid.}
\label{fakecl}
\end{figure}

Note that the bonded closure does not yield a bonded link. The bonded closure can be interpreted as an indication that the two endpoints of the multi-knotoid are attracted by a physical force. We believe that this closure operation with the above interpretation could play an important role in the study of open proteins.

The theory of enhanced bonded links carries through in complete analogy to the bonded multi-knotoids, to obtain the theory of {\it enhanced bonded multi-knotoids}. Only, for the bonded closure it rather makes sense to consider only the bonded closure with an attracting bond.

\subsection{(Enhanced) Bonded braidoids}

In \cite{GL1} braidoid diagrams are defined (similarly to classical braid diagrams), as systems of finite descending strands that involves one or two strands starting with or terminating at an endpoint that is not necessarily at the top or bottom lines of the defining region of the diagram. For details and examples the reader is referred to \cite{GL1}. Here we equip braidoids with bonds.

\begin{definition}\rm
A {\it  braidoid diagram} $B$ is a system of a finite number of arcs immersed in $[0, 1] \times [0,1] \subset \mathbb{R}^2$, where $\mathbb{R}^2$ is identified with the xt-plane, such that the t-axis is directed downward. The arcs of $B$ are the strands of $B$. Each strand  is naturally oriented downward, with no local maxima or minima, following the natural orientation of $[0, 1]$. Moreover, there are only finitely many intersection points among the strands, which are transversal double points endowed with over/under data, and are called crossings of $B$. A  braidoid diagram has two types of strands, the classical strands, i.e. braid strands connecting points on $[0, 1]\times \{0\}$ to points on $[0, 1] \times \{1\}$, and the {\it free strands} that either connect a point in $[0,1]\times \{0\}$ or in $[0,1]\times \{1\}$ to an {\it endpoint} located anywhere in $[0, 1]\times [0, 1]$, or they connect two endpoints that are located anywhere in $[0, 1] \times [0, 1]$. These points that lie on $[0, 1] \times \{0\}$ or $[0, 1] \times \{1\}$, are called {\it braidoid ends}. 

A {\it bonded braidoid diagram}  is a braidoid diagram equipped with bonds, that is, simple horizontal arcs whose endpoints lie on the strands of the bonded braidoid and whose endpoints, called {\it nodes}, differ from the endpoints of the braidoid and from the braidoid ends.
\end{definition}

We now present bonded braidoid isotopy:

\begin{definition}\label{broidiso}\rm
Two bonded braidoid diagrams are said to be {\it isotopic} if one can be obtained from the other by a finite sequence of the moves for bonded braids together with the following moves, all giving rise to the  {\it bonded braidoid isotopy}:
\smallbreak
\begin{itemize}
\item[$\bullet$] {\it Bonded vertical moves} as illustrated in Figure~\ref{biso1}: the endpoints of a bonded braidoid diagram can be pulled up or down in the vertical direction but without letting an endpoint of a bonded braidoid diagram to be pushed/pulled over or under a strand, since this move would correspond to a forbidden move for bonded braidoids (compare with the vertical moves of \cite{GL1} for braidoids).
\smallbreak

\begin{figure}[ht] 
\begin{center} 
\includegraphics[width=3in]{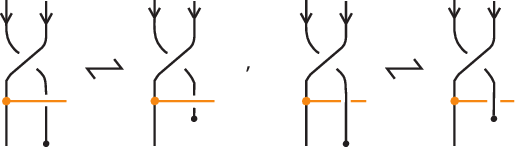} 
\end{center} 
\caption{A vertical move on a bonded braidoid.} 
\label{biso1} 
\end{figure} 

\item[$\bullet$] {\it Bonded swing moves} as illustrated in Figure~\ref{biso2}: the endpoints are allowed to swing to the right or the left like a pendulum as long as the downward orientation on the moving arc is preserved, and the forbidden moves are not violated. A swing move could include a bonded arc or not, see Figure~\ref{biso2}. 
\end{itemize} 
\smallbreak 
\begin{figure}[ht] 
\begin{center} 
\includegraphics[width=3.3in]{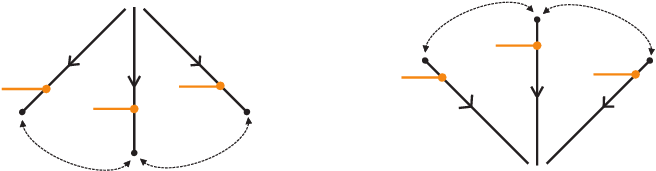} 
\end{center} 
\caption{The swing moves on bonded braidoids.} 
\label{biso2} 
\end{figure} 

\smallbreak
\noindent An isotopy class of bonded braidoid diagrams is called a {\it bonded braidoid}. 
\end{definition}

Our intention now is to adopt the braidoiding algorithm of \cite{GL1} for bonded braidoids. For that we need to define a closure operation for bonded braidoids. We have the following definition: 

\begin{definition}\rm
A {\it labeled bonded braidoid diagram} is a bonded braidoid diagram whose corresponding ends are labeled either with ``o'' or ``u'' in pairs. The {\it closure} of a labeled braidoid is realized by joining each pair of corresponding ends by a vertical segment, either over or under the rest of the bonded braidoid diagram (including the bonds), according to the label attached to the two bonded braidoid ends (see Figure~\ref{clab}).
\end{definition}

\begin{figure}[ht]
\begin{center}
\includegraphics[width=4.8in]{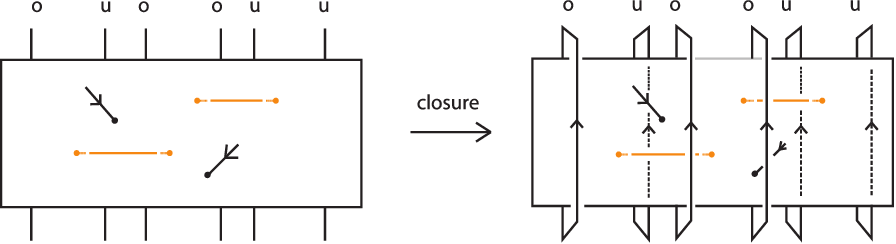}
\end{center}
\caption{The closure of an abstract labeled bonded braidoid.}
\label{clab}
\end{figure} 

In \cite{GL1, GL2}, the authors present a braidoiding algorithm for multi-knotoids, with the use of which, they obtain the analogue of the Alexander theorem for knotoids. In particular, they prove that any (multi)-knotoid diagram is isotopic to the closure of a (labeled) braidoid diagram. The braidoiding algorithm is the braiding algorithm presented in \cite{LR1} (see \S~\ref{alsec} in this paper) adopted in the case of (multi)-knotoids. The reader is referred to \cite{GL1, GL2} for details regarding  braidoiding algorithms for multi-knotoids and how all steps of the braiding algorithm in \cite{LR1} are adapted to the case of knotoids.

As shown in \S~\ref{alsec}, the braiding algorithm for classical braids carries through to the case of topological (enhanced) bonded links by taking care of the orientation of the strands that contain the node of a (enhaned) bond. The same techniques and ideas carry through when it comes to adapting the braidoiding algorithm of \cite{GL1, GL2} to the case of topological (enhanced) bonded multi-knotoids. Hence, we obtain the following  result:

\begin{theorem}[{\bf The analogue of the Alexander theorem for (enhanced) bonded multi-knotoids}] \label{alexkn}
Any oriented topological (enhanced) bonded (multi)-knotoid diagram is isotopic to the closure of a labeled  (enhanced) bonded braidoid diagram.
\end{theorem}

It is crucial to note that different labels on the braidoid  ends of a bonded braidoid may yield non-equivalent closures, due to the presnce of the forbidden moves.  

Moreover, it is worth mentioning that in \cite{GL1} the authors prove that any knotoid diagram may be isotoped to the closure of some labeled braidoid diagram whose labels are all ``u'' (\cite{GL1} Corollary 1), and they define a {\it uniform braidoid} to be a labeled braidoid with all labels ``u''. Similarly, we shall call the labeled bonded braidoid with all labels ``u'', the {\it uniform bonded braidoid} and we have the following result:

\begin{theorem}
Any bonded multi-knotoid diagram is isotopic to the uniform closure of a bonded braidoid diagram.
\end{theorem}

\begin{proof}
We first treat strands of the braids that contain the nodes of bonds as shown in Figures~\ref{orbond} and \ref{orbond1}. In that way we obtain up-arcs that do not contain the nodes of bonds. We apply then the braidoiding algorithm of \cite{GL1} (which is similar to the braidoiding algorithm we presented in \S~\ref{alsec} and we obtain a bonded braidoid, whose closure is isotopic to the bonded multi-knotoid.
\end{proof}

\section{Conclusions and Further Work}\label{sec11:furtherwork} 

In this paper we have used topological, diagrammatic and algebraic methods to define and analyze bonded structures. This means that it is assumed that the reader understands how to translate such language into three dimensions. Diagrams can be regarded as schemata for producing specific three dimensional structures and embeddings, just as a weaving pattern is an instruction for producing a given weave. Our topological argumentation is fully rigorous because the diagrammatics are a formal system for these topological and combinatorial structures. Thus the diagrammatics form a pivotal place where one has formality on the one hand and interpretability on the other. 

This topologists' stance may be new to some scientists who look directly at three dimensional structure. Our approach is particularly useful for the formulation of algebraic invariants and for the formulation of algorithms for computation of invariants. To understand the full story for subjects like protein folding, one needs three dimensional structure and here the diagrams form a basis for further articulation.
If we wish to further study physical interactions (beyond the present paper), then one can add more structure to the combinatorial models given here and work with them three dimensionally. This is a project for further research, and it promises deeper relationships between our invariants and the physical behaviours of molecules.

The combinatorial and algebraic coding of structures that we use in this paper has many potentials for applications. The standard diagramming can be used to produce embeddings in three dimensional space directly. The braid representations are concise algebraic methods to encode bonded structures and they are new and need to be studied further for their potential. Each method of formalizing a three dimensional topological structure has its own properties that deserve further research. 

In a paper to follow the present work we will investigate more deeply the invariants described here. These include the unplugging invariants, the invariants using insertion and the Kauffman bracket and Jones polynomials, and the new invariants that arise from the $L$-move braiding formulations for bonded knots and links. In the case of the $L$-moves, we will continue the formulation in terms of generalizations of the algebraic Markov Theorem for bonded braids and consequent invariants defined in these terms.

Computational approaches to our work include the use of molecular databases  (we refer to Sulkowska's works in KnotProt and the Protein Data Base \cite{DRGDSMRSS}) where one can translate three dimensional data into the combinatorics, allowing one to compute invariants and analyse the bonding structure of proteins. In the other direction one can translate diagrammatic encoding into chosen three dimensional embeddings and then work with these models in three dimensional space, including comparison with molecules from the database. This relationship with experimental and three dimensional information is ongoing. We have used these methods in our previous work and will continue to use it in applications of the present research. 

In the case of the ideas we have suggested for Feynman diagrams, the ideas are in a state of flux in that it is not yet clear that adding knotted structure to Feynman diagrams will advance the understanding of quantum field theory. On the other hand, field theoretic approaches to protein folding are clearly needed (and under investigation \cite{N}) and we need to see how such approaches are related to the three dimensional combinatorial and topological structure of molecules. The Feynman diagrams are an intermediation between the field theory and the combinatorics. For all these reasons we regard the present work as constructing a foundation for much-needed further research.

 Bonded knotoids are especially relevant for modeling open chains such as proteins. We
will revisit (enhanced) bonded knotoids and braidoids  and their topological interactions in the form of (enhanced) bonded braidoid equivalences.

 In another direction we will  define the plat closure for bonded braids and braidoids and will formulate Hilden and Birman type theorems for turning an unoriented bonded knot/knotoid to bonded plat/platoid and for their equivalences. 

We will explore a bonded Morse category, a natural next step for multi-knotoids and linkoids: a Morse-theoretic diagrammatics that admits cups/caps and incorporates bonds. This framework will carry a corresponding move calculus and normal-form results, enabling functorial invariants via monoidal functors to module categories (with trace constructions).


\end{document}